\documentclass[11pt,a4paper]{amsart}

\usepackage{amssymb} 
\usepackage[all]{xy} 
\usepackage{graphicx} 
\usepackage{eucal} 
\usepackage{hyperref} 
\usepackage{color} 

\numberwithin{equation}{section}

\newtheorem{definition}{Definition}[section]
\newtheorem{lemma}[definition]{Lemma}
\newtheorem{theorem}[definition]{Theorem}
\newtheorem{prop}[definition]{Proposition}

\theoremstyle{remark}
\newtheorem{example}[definition]{Example}
\newtheorem{remark}[definition]{Remark}

\newenvironment{pf}{\begin{proof}\upshape}{\hfill
\end{proof}} 

\newcommand{\R}{\mathbb{R}} 
\newcommand{\lcf}{\lbrack\! \lbrack} 
\newcommand{\rcf}{\rbrack\! \rbrack} 
\newcommand{\g}{\mathfrak{g}}

\newcommand{\X}{\mathfrak{X}}
\renewcommand{\L}{{\mathcal L}}
\renewcommand{\d}{\mathrm d}
\newcommand{\smalcirc}{\mbox{\,\tiny{$\circ $}\,}}  

\def\beq{\begin{equation}}
\def\eeq{\end{equation}}
\def\bea{\begin{eqnarray}}
\def\eea{\end{eqnarray}}
\def\beann{\begin{eqnarray*}}
\def\eeann{\end{eqnarray*}}
\def\beasn{\begin{sneqnarray}}
\def\eeasn{\end{sneqnarray}}
\def\ben{\begin{enumerate}}
\def\een{\end{enumerate}}
\def\bit{\begin{itemize}}
\def\eit{\end{itemize}}

\def\r{\ensuremath{\mathbb{R}}}
\def\rk{{\mathbb R}^{k}}

\def\tkqh{(T^1_k)^*Q}

\def\d{{\rm d}}

\def\rk{{\mathbb R}^{k}}
\def\ga{\mathfrak{g}}
\def\g^*{\mathfrak{g}^*}

\def\qed{\ifvmode\Realemovelastskip\fi
{\unskip\nobreak\hfil\penalty50\hbox{}\nobreak\hfil \hbox{\vrule
height1.2ex width1.2ex}\parfillskip=0pt \finalhyphendemerits=0
\par\smallskip}}

\def\qed{\ifvmode\removelastskip\fi
{\unskip\nobreak\hfil\penalty50\hbox{}\nobreak\hfil \hbox{\vrule
height1.2ex width1.2ex}\parfillskip=0pt \finalhyphendemerits=0
\par\smallskip}}

\title{Poly-Poisson structures}
\author[D. Iglesias Ponte]{David Iglesias}
\address{David Iglesias:
    Departamento de Matem\'{a}tica Fundamental,
    Facultad de Matem\'{a}ticas, Universidad de la Laguna, Spain}
    \email{diglesia@ull.es}

\author[J.C. Marrero]{Juan Carlos Marrero}
\address{Juan Carlos Marrero:
    Departamento de Matem\'{a}tica Fundamental,
    Facultad de Mate\-m\'{a}ticas, Universidad de la Laguna, Spain}
    \email{jcmarrer@ull.es}

\author[M. Vaquero]{Miguel Vaquero}
\address{Miguel Vaquero:
    Instituto de Ciencias Matem\'aticas (CSIC-UAM-
UC3M-UCM), Spain}
    \email{miguel.vaquero@icmat.es}

\keywords{polysymplectic manifolds, poly-Poisson structures,
  polysymplectic foliation, reduction, Lie algebroids, linear Poisson structures}

 \subjclass[2010]{53D05, 53D17, 70645}

\begin{document}

\begin{abstract}
In this paper we introduce poly-Poisson structures as a higher-order extension of Poisson structures. It
is shown that any poly-Poisson structure is endowed with a polysymplectic foliation.
It is also proved that if a Lie group acts polysymplectically on a
polysymplectic manifold then, under certain regularity conditions, the
reduced space is a poly-Poisson manifold. In addition, some
interesting examples of poly-Poisson manifolds are discussed.
\end{abstract}
\maketitle

\tableofcontents


\section{Introduction}
As it is well known, symplectic manifolds play a fundamental role
in the Hamiltonian formulation of classical mechanics. Indeed, 
given a configuration space $Q$, its cotangent bundle $T^*Q$, 
the phase space, is endowed with a symplectic structure and the
Hamilton's equations associated to a Hamiltonian $H\colon T^*Q
\to \r$ can be interpreted as the flow of the Hamiltonian
vector field associated with $H$.

Under the existence of symmetries, it is possible
to perform a reduction procedure in which some of the variables
are reduced.
The interest of the reduction procedure is twofold: not only we 
are able to reduce the dynamics but it is a way to generate new
examples of symplectic manifolds. One of the reduction methods
is the Marsden-Weinstein-Meyer reduction procedure \cite{MW-1974}: Given 
a Hamiltonian action 
of a Lie group $G$ on a symplectic
manifold $(M,\Omega )$ with equivariant moment map $J\colon M\to \g^*$, 
it is possible to obtain a symplectic structure on the quotient
manifold $J^{-1}(\mu )/G_{\mu}$. An interesting example is the action
of a Lie group $G$ on its cotangent bundle $T^*G$ by cotangent lifts
of left translations and the moment map $J\colon T^*G\to \g^*$ is
given by $J(\alpha _g)= (T_eR_g)^*(\alpha _g)$. Here, the reduced space 
$J^{-1}(\mu )/G_{\mu}$ is just the coadjoint orbit along the element
$\mu\in \g^*$ endowed with the Kirillov-Kostant-Souriau symplectic
structure.

On the other hand, the existence of a Lie group of symmetries for a symplectic 
manifold is one of the justifications for the introduction of Poisson 
manifolds, which generalize symplectic manifolds. Indeed if $(M,\Omega )$ 
is a symplectic manifold and $G$ is a Lie group acting freely and properly on $M$ and 
preserving the symplectic structure then the quotient manifold $M/G$ 
is endowed with a Poisson structure in such a way that 
the projection $\pi \colon M\to M/G$ is a Poisson epimorphism. 
A particular example is the so-called Poisson-Sternberg structure
on $T^*Q/G$, which is obtained reducing the symplectic structure on $T^*Q$
after lifting a free and proper action of a Lie group $G$ on $Q$ to $T^*Q$ 
(see \cite{MoMaRa} for applications in Mechanics). Any Poisson
structure induces a foliation with symplectic leaves, the symplectic
foliation. For symplectic manifolds there is only one leaf: the manifold
itself. The other interesting example is the Lie Poisson structure on $\g^*$, 
the dual space of any Lie algebra $\ga$. In this case, the symplectic leaves
are the coadjoint orbits previously obtained as reduced symplectic
spaces.

In \cite{Gunther-1987}, polysymplectic manifolds were introduced as a natural 
vector-valued extension of symplectic manifolds. A
polysymplectic structure on a manifold $M$ is a closed 
nondegenerate $\rk$-valued 2-form on $M$. The physical motivation
was obtain a geometric framework in which develop field theories
for Lagrangians defined on the first-jet bundle of a trivial
fibration $p \colon M=N\times U\to U$, $U\subseteq \rk$ being
the parameter space of the theory. In this case, the natural polysymplectic
manifold is $LQ$, the tangent frame bundle of $Q$. More generally,
it is possible to consider the cotangent bundle of $k$-covelocities 
$(T^1_k)^*Q:=T^{*}Q \oplus \stackrel{(k}{\dots}\oplus T^{*}Q$. 

In addition, in \cite{Gunther-1987} the author considers 
polysymplectic actions of Lie groups on polysymplectic 
manifolds and the notion of a polysymplectic moment map $J\colon M\to
\mathfrak{g}^*\otimes \mathbb{R}^k$ is introduced. We note that 
other moment map theories with different target spaces (such 
as Lie groups or homogeneous spaces \cite{AK,AKM,AMM}) have
been proposed but this one fits in a different framework.
In addition, in this situation, it is developed a reduction procedure
mimicking the symplectic case and, under certain conditions, a polysymplectic structure on
$J^{-1}(\mu)/G_{\mu}$ is obtained. We must remark here that, the
assumed conditions in \cite{Gunther-1987} do not assure that the
quotient space $J^{-1}(\mu)/G_{\mu}$ is polysymplectic because this
fact is based on a wrong result (see Lemma 7.5 in
\cite{Gunther-1987}). In \cite{MSV}, a right version of the
Marsden-Weinstein reduction theorem in the polysymplectic setting is proved.

As an example of the reduction procedure, a polysymplectic
structure on $k$-coadjoint orbits $\mathcal{O}_{(\mu_1,\ldots, \mu_k)}$, 
for $(\mu_1,\ldots, \mu_k)\in \g^*\otimes\rk=\g^*\times\stackrel{(k}{\ldots}
\times\g^*$ is obtained. 
A natural question arises: Is there a geometric structure on 
$\g^*\times\stackrel{(k}{\ldots}\times \g^*$ which admits a
polysymplectic foliation whose leaves are the $k$-coadjoint orbits?

In this paper, it is proposed an answer to this question introducing 
the notion of a poly-Poisson structure as a generalization of Poisson 
structures as well as of polysymplectic structures. Contrary to what
would be expected, these are not just $\rk$-valued 2-vector fields.
Indeed, if one starts from a polysymplectic structure $\bar{\omega}$,
one can construct a vector-bundle morphism $\bar{\omega}^{\flat}\colon TM \rightarrow (T^1_k)^*M$
and the non-degeneracy of $\bar{\omega}$ is equivalent to $\bar{\omega}^\flat$
being injective (see Proposition \ref{poly2}). However, if $k>2$ this does not imply
that $\bar{\omega}^\flat$ is an isomorphism, so it does not make sense
to take the inverse of this map. Our approach is the following one: if
$S:=\textrm{Im }(\bar{\omega}^{\flat})$ then one can consider the inverse
of $\bar{\omega}^{\flat}$ restricted to $S$, that is,
\[
\bar{\Lambda}^{\sharp}:=(\bar{\omega}^{\flat})^{-1}_{|S} \colon S\to TM.
\]
Thus, this fact suggests to introduce the notion of a poly-Poisson structure
as a pair $(S,\bar{\Lambda}^{\sharp})$, where $S$ is a vector subbundle of
$(T^1_k)^*M$ and $\bar{\Lambda}^{\sharp}\colon
S\rightarrow TM$ is a vector bundle morphism satisfying several conditions
(see Definition \ref{def:poly-poisson}).
As for Poisson manifolds, any poly-Poisson manifold
is endowed with a foliation, whose associated distribution
is $\bar{\Lambda}^\sharp (S)$, with polysymplectic leaves.
Moreover, a second motivation from Poisson geometry is
the following one. If $G$ is a Lie group which acts 
polysymplectically on a polysymplectic manifold then it is possible to obtain,
under some mild regularity conditions, a poly-Poisson structure
on the reduced space $M/G$. 

The paper is organized as follows. In section 2, it is recalled the notion of 
a polysymplectic structure and it is obtained a characterization in terms
of bundle maps. After giving some examples, in section 3 we introduce the
notion of a poly-Poisson structure. After that it is shown that any 
poly-Poisson manifold has associated a generalized foliation and that on each leaf it
is induced a polysymplectic structure. In order to obtain examples, 
it is proposed a method to construct poly-Poisson structures out of a family
of Poisson structures. As an application, given a linear Poisson structure
on a vector bundle $E^*$ with base space $Q$, it is described a poly-Poisson structure 
on the Whitney sum $E^*\oplus \ldots \oplus E^*$ and, in particular, on
the frame bundle $LE\to Q$. In section 5, it is developed a reduction
procedure for polysymplectic manifolds. More precisely, if $G$ is a Lie
group acting polysymplectically on a polysymplectic manifold $(M,\Omega )$ then, under certain
regularity conditions, the quotient manifold $M/G$ is endowed with a 
poly-Poisson structure. As an application, given a principal
$G$-bundle $Q\to Q/G$ with total space $Q$ it is possible to obtain a
poly-Poisson structure on the Withney sum $T^*Q/G \oplus
\stackrel{(k}{\dots}\oplus T^{*}Q/G$ as the reduction by the Lie group
$G$ of the cotangent bundle of $k$-covelocities $T^{*}Q \oplus
\stackrel{(k}{\dots}\oplus T^{*}Q$. As a consequence, we deduce that
the reduction of the tangent frame bundle $LQ\to Q$ is the frame
bundle of the Atiyah algebroid $TQ/G\to Q/G$. The paper ends with and
appendix which contains some basic definitions and results on Lie
algebroids and fiberwise linear Poisson structures on vector bundles.

We finally remark that there is a relation with Dirac structures (see \cite{Co}). 
Dirac structures provide a common framework to Poisson and presymplectic manifolds. These
are integrable Lagrangian subbundles of $TM\oplus T^*M$. To interpret poly-Poisson
structures in this setting, we could consider the bundle $TM\oplus (T^1_k)^*M$, which 
is endowed with a nondegenerate symmetric bilinear form with values in $\rk$ \cite{L}. We postpone to a 
future work the details of this, hoping it will provide a framework to analyze reduction (see,
for instance, \cite{BC, BC2}).

\textbf{Notation:} If $Q$ is a differentiable manifold of dimension $n$, $(T^1_k)^*Q$ will denote 
the cotangent bundle of $k$-covelocities of $Q$, that is, the Whitney sum of $k$-copies of the 
cotangent bundle,
$(T^1_k)^*Q:=T^{*}Q \oplus \stackrel{(k}{\dots}\oplus T^{*}Q$. For an element of $(T^1_k)^*Q$, the 
symbol $\ \bar{} \ $ will be used to denote it, and subscripts without $\ \bar{} \ $ to denote its 
components, i.e.
$\bar{\alpha}\in(T^1_k)^*Q$ and $\bar{\alpha}=(\alpha_1,\ldots,\alpha_k)$ where $\alpha_i\in T^*Q$ 
for $i\in\{1,\ldots,k\}$. The same criterion will be used for 2-forms and morphisms. If
$\bar{\alpha}\in(T^1_k)_{p}^*Q$ are $k$ 1-forms over the same point, $X\in T_pQ$ is a tangent vector and
$W\subset T_pQ$ is a vector subspace, we write $\bar{\alpha}(X)=(\alpha_1(X),\ldots,\alpha_k(X))\in\mathbb{R}^k$
for the evaluation and the restriction $\bar{\alpha}_{|W}=(\alpha_{1|W},\ldots,\alpha_{k|W})\in W^*\oplus \ldots
\oplus W^*$ is just the restriction of every component. If $\bar{\alpha}$ is a section of $(T^1_k)^*Q$ and $X$ a
vector field we will write $\mathcal{L}_X\bar{\alpha}=(\mathcal{L}_X\alpha_1,\ldots,\mathcal{L}_X\alpha_k)$ and
if $\bar{f}$ is a $\mathbb{R}^k$-valued function, $\bar{f}=(f_1,\ldots,f_k)$, then we will write $\d\bar{f}=(\d
f_1,\ldots,\d f_k)$. When we deal with the fibers of bundles over a manifold, we sometimes omit reference to
base point when it is clear or when what are we saying is true for all points of the base manifold.
\section{Polysymplectic structures}\label{poly_manifolds}
In this section, it is recalled the concept of a polysymplectic structure, which is one of the possible higher
order analogues of a symplectic structure (see \cite{Gunther-1987}).

\begin{definition}\label{poly1}
A {\rm $k$-polysymplectic structure} on a manifold $M$ is 
a closed nondegenerate $\rk$-valued $2$-form
\[
\bar{\omega}=\displaystyle\sum_{A=1}^k \omega^A\otimes e_A
\]
where $\{e_1,\ldots, e_k\}$ denotes the canonical basis of $\rk$.

Equivalently, $\bar{\omega}$ can be seen as a family of $k$ closed $2$-forms $(\omega^1,\ldots, \omega^k)$ such that
\begin{equation}\label{poly-cond}
\displaystyle\cap_{A=1}^k {\rm Ker\,}\omega^A=0\,.
\end{equation}
The pair $(M,\bar{\omega})=(M,\omega^1,\ldots, \omega^k)$ is called a
{\rm $k$-polysymplectic manifold} or simply a {\rm polysymplectic manifold}.
\end{definition}

\begin{remark}
It is clear that the notion of a 1-polysymplectic structure coincides with the definition of a symplectic structure.
\end{remark}


Polysymplectic structures can be thought of as vector bundle morphisms, obtaining an equivalent definition of
polysymplectic manifolds.
\begin{prop}\label{poly2}
A {\rm $k$-polysymplectic structure} on a manifold $M$ is a vector bundle morphism
$\bar{\omega}^{\flat}\colon TM \rightarrow (T^1_k)^*M$ of the tangent bundle $TM$ to $M$ on the
cotangent bundle of $k$-covelocities of $M$ satisfying the conditions:
\begin{itemize}
\item[i)] Skew-symmetry:
\[
{\bar{\omega}}^{\flat}(X)(X)=(0,\ldots,0)  ,\qquad  X \in T_pM, \ p\in M.
\]
\item[ii)] Non degeneracy ($\bar{\omega}^{\flat}$ is a monomorphism):
\[
\textrm{Ker }(\bar{\omega}^{\flat})=\{0\}.
\]
\item[iii)] Integrability condition:
\begin{equation}\label{eq:integrability:polysymplectic}
{\overline{\omega}}^{\flat}([X,Y])={\mathcal{L}}_X{\overline{\omega}}^{\flat}(Y)-
{\mathcal{L}}_Y{\overline{\omega}}^{\flat}(X)+\d({\overline{\omega}}^{\flat}(X)(Y)),  \qquad X, Y \in \X(M).
\end{equation}
\end{itemize}
\end{prop}

\begin{proof}
Let $\bar{\omega}^{\flat}\colon TM \rightarrow (T^1_k)^*M$ be a vector bundle morphism (over the identity of $M$).
Then, condition i) is equivalent to the existence of $k$ 2-forms $\omega^1, \ldots, \omega^k$ on $M$ such that
\[
\bar{\omega}^\flat (X)=((\omega^1)^\flat (X),\ldots ,(\omega^k)^\flat (X)), \qquad \textrm{ for } X\in TM.
\]
On the other hand, ii) holds if and only if
\[
\displaystyle\cap_{A=1}^k {\rm Ker\,}\omega^A=\{0\}.
\]
Finally, using the following relation
\[
i_Yi_X\d\Phi + \Phi^\flat [X,Y]=\L _X\Phi^\flat (Y)-\L _Y\Phi^\flat (X)+\d(\Phi^\flat (X)(Y)),\qquad \textrm{ for } X,Y\in \X(M),
\]
which is valid for any arbitrary 2-form $\Phi$ on $M$, we deduce that \eqref{eq:integrability:polysymplectic} holds if and only if $\omega ^A$ is a closed 2-form, for all $A$. This proves the result.
\end{proof}
Next, an interesting example of a polysymplectic manifold will be presented.
\begin{example}[The tangent frame bundle \cite{N}]\label{ex:frame:bundle}

Let $LM$ be the tangent frame bundle of $M$, that is, the elements of $LM$ are pairs $(m,\{e_i\})$, where $\{ e_i\}$, $i=1,\ldots n$, is a linear frame at $m\in M$. There is a right action of $GL(n)$, $n=\textrm{dim }M$,
on $LM$ by
\[
(m,e_i)\cdot g=(m,e_i g_{ij}), \qquad \textrm{ for } (m,e_i)\in LM, g=(g_{ij})\in GL(n).
\]
With this action $LM\to M$ is a principal $GL(n)$-bundle.

On $LM$ there exists a canonical vector-valued $1$-form $\vartheta
\in \Omega^1(LM,\r^n)$, called the \textit{soldering one-form},  which is defined by
\[
 \begin{array}{rccl}
 \vartheta(u)\colon & T_u(LM) & \to & \r^n\\\noalign{\medskip} & X & \mapsto & \vartheta(u)(X)=u^{-1}(T_u\pi (X))  ,
\end{array} 
\]
where $\pi\colon LM\to M$ is the canonical projection and $u=(m,e_i)\in LM$ is seen as the non-singular linear map $u\colon \r^n\to T_mM$, $(v^1,\ldots ,v^n)\mapsto \sum_{i=1}^n v^ie_i$.

The soldering one-form endowes $LM$ with a $n$-polysymplectic structure given by
\[
  \omega =-d\vartheta .
\]
(for more details, see \cite{N}).\hfill$\diamond$

\end{example}

Next, it will be shown a method to construct a polysymplectic structure out of a family of symplectic manifolds.
Indeed, this construction is the model for a regular polysymplectic manifold.
\begin{example}\label{ex:regular:polysymplectic}
Let $(M_A,\Omega ^A)$ be a symplectic manifold, with $A\in\{1,\ldots ,k\}$, and $\pi_A\colon M\to M_A$
a fibration. Denote by $\omega ^A$ the presymplectic form on $M$ given by
\[
\omega ^A=\pi_A^*(\Omega ^A), \textrm{ for } A\in \{1,\ldots ,k\}.
\]
It is easy to prove that $\omega ^A$ has constant rank. In fact,
\[
\textrm{Ker }\omega^A=\textrm{Ker }(T\pi _A),
\]
where $T\pi _A\colon TM\to TM_A$ is the tangent map to $\pi _A\colon M\to M_A$. Thus, we deduce that
the family of 2-forms $(\omega ^1,\ldots ,\omega ^k)$ on $M$ define a $k$-polysymplectic structure
if \( \cap _{A=1}^k \textrm{Ker }(T\pi _A)=\{ 0\}. \)

Conversely, if $(\omega ^1,\ldots ,\omega ^k)$ is a $k$-polysymplectic structure on a manifold $M$ and the 2-forms $\omega^A$ have constant rank then the distributions ${\mathcal F}_A$, with $A\in\{1,\ldots ,k\}$,
given by
\[
x\in M \to {\mathcal F}_A (x)=\textrm{Ker }(\omega ^A(x))\subseteq T_xM
\]
are completely integrable (this follows using that $\omega ^A$ is closed). In addition, if the space of leaves
$M_A=M/{\mathcal F}_A$ of the foliation ${\mathcal F}_A$ is a quotient manifold, for $A\in\{1,\ldots ,k\}$,
and $\pi _A\colon M\to M_A=M/{\mathcal F}_A$ is the canonical projection, we have that the 2-form $\omega ^A$ is basic
with respect to $\pi _A$. Therefore, there exists a unique 2-form $\Omega ^A$ on $M_A$ satisfying
\[
\omega ^A=\pi _A^*(\Omega ^A).
\]
Finally, since the map $\pi _A^*\colon \Omega ^p(M_A)\to \Omega ^p(M )$ is a monomorphism
of modules, we deduce that $\Omega ^A$ is a symplectic form on $M_A$.\hfill$\diamond$
\end{example}

Some particular examples of the previous general construction are the following ones.
\begin{example}[The product of symplectic manifolds]\label{ex:product:symplectic}
Let $(M_A,\Omega ^A)$ be a symplectic manifold, with $A\in\{1,\ldots ,k\}$, and $M=M_1\times \ldots \times M_k$.
Then, it is clear that $M$ admits a $k$-polysymplectic structure $(\omega ^1,\ldots, \omega ^k)$. It is
sufficient to take
\[
\omega ^A=\textrm{pr}_A^*(\Omega ^A),\qquad \textrm{ for }A\in \{1,\ldots ,k\},
\]
where $\textrm{pr}_A\colon M\to M_A$ is the canonical projection from $M$ on $M_A$.
\end{example}

\begin{example}[The cotangent bundle of $k$-covelocities of a manifold]\label{canexample}
It is well known that, given an arbitrary manifold $Q$, its cotangent bundle $T^*Q$ is endowed with a canonical
symplectic structure in the following way. If $\theta$ is the Liouville 1-form on $T^*Q$ then $\omega=-\d\theta$
is a symplectic 2-form on $T^*Q$ (see, for instance, \cite{AM-1978,LM-1987}).

Next, using this symplectic structure and the previous general construction, 
we will obtain a $k$-polysymplectic structure on the cotangent bundle of $k$-covelocities $(T^1_k)^*Q:=T^{*}Q
\oplus \stackrel{(k}{\ldots}\oplus T^{*}Q$ of $Q$. Denote by $\pi^{k}_Q\colon (T^1_k)^*Q\to Q$ the canonical
projection and by $\pi^{k,A}_Q\colon (T^1_k)^*Q\to T^*Q$ the fibration defined by
\[
\pi_Q^{k,A}(\alpha_1,\ldots,\alpha_k)=\alpha_A,\qquad \textrm{for }A\in\{1,\ldots k\}.
\]
If $(q^i)$ are local coordinates on $U\subseteq Q$, then the induced local coordinates $(q^i,p_i^A)$
on $(\pi ^k_Q)^{-1}(U)$ are given by
\[
q^i(\alpha^q_1,\ldots,\alpha^q_k)=q^i(q), \quad p_i^A(\alpha^q_1,\ldots,\alpha^q_k)=\alpha_A^q(\frac{\partial}{\partial q^i}_{|q}), \quad 1\leq i\leq n; 1\leq A\leq k.
\]
Moreover, the local expression of $\pi_Q^{k,A}$ is
\[
\pi_Q^{k,A}(q^i,p_i^1,\ldots ,p_i^k)=(q^i,p_i^A).
\]
Then, it is clear that
\[
\cap _{A=1}^k \textrm{Ker }(T\pi _Q^{k,A})=\{0\}.
\]
Therefore, $(\omega ^1,\ldots ,\omega ^k)$ is a $k$-polysymplectic structure on $(T^1_k)^*Q$,
where $\omega ^A$ is the 2-form on $(T^1_k)^*Q$ defined by
\[
\omega ^A=(\pi_Q^{k,A})^*\omega .
\]
Note that the local expression of $\omega ^A$ is
\begin{equation}\label{locexp}
\omega^A = \displaystyle  dq^i \wedge dp^A_i.
\end{equation}
\end{example}
\begin{remark}\label{rmk:relation}
Let $Q$ be a manifold of dimension $n$. Then, from Example \ref{ex:frame:bundle}, the tangent frame bundle $LQ$ is endowed with a polysymplectic structure $-\d\vartheta$. On the other hand, we have just seen that the cotangent bundle of $n$-covelocities, $(T^1_n)^*Q$, also is a polysymplectic manifold.

In addition, $LQ$ is an open subset of $(T^1_n)^*Q$, where the inclusion 
$\iota \colon
LQ\to (T^1_n)^*Q$ is given by $u\mapsto (\textrm{pr}_1\circ u^{-1}, \ldots ,\textrm{pr}_n\circ u^{-1})$. Note that, $\pi =\pi ^n_Q \circ \iota$, where $\pi\colon LQ\to Q$ and $\pi^n_Q\colon (T^1_n)^*Q\to Q$ are the projections. Let us see that the restriction of the poly-symplectic structure on $(T^1_n)^*Q$ is just $\d\vartheta$. More precisely, let us show that
\[
\iota ^* \left ( \displaystyle\sum_{A=1}^n (\pi _Q^{n,A}){}^* (\theta ) \otimes e_A \right )= \vartheta ,
\]
where $\{e_1,\ldots, e_n\}$ denotes the canonical basis of $\R ^n$, $\theta$ is the Liouville 1-form, $\pi
_Q^{n,A}\colon (T^1_n)^*Q\to T^*Q$ is the projection on the $A$-factor, for $A\in\{1,\ldots n\}$,  and $\vartheta$ is the soldering 1-form. If $(X_1,\ldots ,X_n)\in T
((T^1_n)^*Q)$ then
\[
\begin{array}{rcl}
\left ( \iota^* (\displaystyle\sum_{A=1}^k (\pi _Q^{n,A})^* (\theta ) \otimes e_A ) (u) \right ) (X_1,\ldots
,X_n)& = &\\[10pt] &\kern-400pt= &\kern-200pt \left ( \displaystyle\sum_{A=1}^k (\pi _Q^{n,A})^* (\theta ) \otimes e_A \right ) (\textrm{pr}_1\circ u^{-1},\ldots
,\textrm{pr}_n\circ u^{-1})(X_1,\ldots ,X_n)
\\[10pt] &\kern-400pt= &\kern-200pt(\textrm{pr}_1\circ u^{-1} (T\pi _Q
(X_1)),\ldots
,\textrm{pr}_n\circ u^{-1} (T\pi _Q (X_n)) ) = u^{-1} (T\pi (X_1,\ldots ,X_n)) \\[10pt]
&\kern-400pt =  &\kern-200pt \vartheta (u) (X_1,\ldots, X_n).
\end{array}
\]
\end{remark}

\begin{example}[$k$-coadjoint orbits \cite{Gunther-1987,MSV}]\label{ex:k-coadjoint orbit}
In this example, it is described a $k$-polysymplectic structure on every $k$-coadjoint orbit associated with a
Lie group $G$. In the particular case when $k=1$, it is recovered the standard symplectic structure on the
coadjoint orbits of $G$.

We will recall the definition of such a structure. Let $G$ be a Lie group and $\mathfrak{g}$ be its Lie algebra. Consider the \textit{coadjoint action}
 \[
 \begin{array}{lccl}
                    \textrm{Coad}\colon & G\times \mathfrak{g}^* &\to & \mathfrak{g}^*\\\noalign{\medskip}
                            & (g,\mu) & \mapsto & \textrm{Coad}(g,\mu)=\mu\circ Ad_{g^{-1}}.
\end{array}
\]
The orbit of $\mu \in\mathfrak{g}^*$ in $\mathfrak{g}^*$ under this action, denoted by
\[
                \mathcal{O}_\mu=\{\textrm{Coad}(g,\mu) \ \mid \  g\in G\}\,,
\]
is equipped with a symplectic structure $\omega_{\mu}$ defined by
\begin{equation}\label{symp-orbit}
\omega_{\mu}(\nu)\left(\xi_{\mathfrak{g}^*}(\nu),\eta_{\mathfrak{g}^*}(\nu)\right)=-\nu [\xi,\eta]
\end{equation}
where $\nu$ is an arbitrary point of $\mathcal{O}_{\mu}$, $\xi,\eta \in\mathfrak{g} $ and
$\xi_{\mathfrak{g}^*}(\nu),\eta_{\mathfrak{g}^*}(\nu)$ are the infinitesimal generators of the coadjoint action associated with $\xi$ and $\eta$ at the point $\nu$ (for more details see,
for instance, \cite[p. 303]{AM-1978}).

The extension to the polysymplectic setting is as follows. Define an action of $G$ over
$\mathfrak{g}^*\times\stackrel{(k}{\ldots}\times\mathfrak{g}^*$ by
\begin{equation}\label{coad^k}
\begin{array}{rcl}
\textrm{Coad}\,^k\colon  G\times \mathfrak{g}^*\times\stackrel{(k}{\ldots}\times\mathfrak{g}^*
                    &\to & \mathfrak{g}^*\times\stackrel{(k}{\ldots}\times\mathfrak{g}^*\\\noalign{\medskip}
                     (g,\mu_1,\ldots , \mu_k) & \mapsto & \textrm{Coad}\,^k(g,\mu_1,\ldots, \mu_k)=\left( \textrm{Coad}\,(g,\mu_1),
                                                                                          \ldots, \textrm{Coad}\,(g,\mu_k)\right)
\end{array}
\end{equation}
which is called the \textit{$k$-coadjoint action}.

Let $\mu=(\mu_1,\ldots, \mu_k)$ be an element of $\mathfrak{g}^*\times\stackrel{(k}{\ldots}\times\mathfrak{g}^*$. We will denote by $\mathcal{O}_\mu=\mathcal{O}_{(\mu_1,\ldots, \mu_k)}$
the $k$-coadjoint orbit at the point $\mu=(\mu_1,\ldots, \mu_k)$. If $\xi\in\mathfrak{g}$ then it is easy
to prove that the infinitesimal generator of $\textrm{Coad}\,^k$ associated with $\xi$ is given by
\begin{equation}\label{eq:inf:k:coadjoint}
\xi_{(\mathfrak{g}^*\times\stackrel{(k}{\ldots}\times\mathfrak{g}^*)}(\nu _1,\ldots ,\nu_k)=
(\xi_{\mathfrak{g}^*}(\nu _1),\ldots ,\xi_{\mathfrak{g}^*}(\nu _k)),\;\; \textrm{ for } (\nu_1,\ldots ,\nu _k)
\in \mathfrak{g}^*\times\stackrel{(k}{\ldots}\times\mathfrak{g}^*.
\end{equation}
On the other hand, the canonical projection
\[
\textrm{pr}_A\colon \mathfrak{g}^*\times\stackrel{(k}{\ldots}\times\mathfrak{g}^*\to \mathfrak{g}^*, \qquad (\nu _1,\ldots ,\nu _k)\mapsto \nu_A,
\]
induces a smooth map $\pi _A\colon \mathcal{O}_\mu=\mathcal{O}_{(\mu_1,\ldots, \mu_k)} \to \mathcal{O}_{\mu_A}$ between the $k$-coadjoint orbit $\mathcal{O}_\mu$ and the coadjoint
orbit $\mathcal{O}_{\mu_{A}}$. Moreover, using \eqref{eq:inf:k:coadjoint}, we deduce that $\pi_A$ is a surjective submersion and, in addition,
\[
\cap _{A=1}^k\textrm{Ker }(T\pi _A)=\{0\}.
\]
Now, we consider the family of 2-forms $(\omega_\mu^{1},\ldots, \omega_\mu^k)$ on $\mathcal{O}_\mu=\mathcal{O}_{(\mu_1,\ldots, \mu_k)}$ given by
\[
\omega_\mu^A = \pi_A^*(\omega_{\mu_A}), \qquad \textrm{ for every }A\in\{1,\ldots ,k\},
\]
where $\omega_{\mu_A}$ is the symplectic 2-form on $\mathcal{O}_{\mu_A}$. Then, $(\omega_\mu^{1},\ldots,
\omega_\mu^k)$ is a $k$-polysymplectic structure on $\mathcal{O}_\mu=\mathcal{O}_{(\mu_1,\ldots, \mu_k)}$  (for
more details, see \cite{Gunther-1987,MSV}).\hfill$\diamond$
\end{example}

\section{K-poly-Poissson structures}

In this section, we will generalize the notion a k-polysymplectic structure dropping
certain non-degeneracy assumption, as Poisson structures generalize symplectic structures.
More precisely,

\begin{definition}\label{def:poly-poisson}
A {\rm$k$-poly-Poisson structure} on a manifold $M$ is a couple $(S,\bar{\Lambda}^{\sharp})$, where $S$ is a
vector subbundle of $({T^1_k})^*M=T^*M\oplus\stackrel{(k}{\ldots}\oplus T^*M$ and $\bar{\Lambda}^{\sharp}\colon 
S\rightarrow TM$ is a vector bundle morphism which satisfies the following conditions:
\begin{itemize}
\item[i)] \label{anti} $\bar{\alpha} (\bar{\Lambda}^{\sharp}(\bar{\alpha}))=0$ for $\bar{\alpha}
\in S$.
\item[ii)]\label{nodege} If $\bar{\alpha}(\bar{\Lambda}^{\sharp}(\bar{\beta}))=0$ for
every $\bar{\beta}\in S$ then $\bar{\Lambda}^{\sharp}(\bar{\alpha})=0$.
\item[iii)] If $\bar{\alpha}, \bar{\beta}\in \Gamma(S)$ are sections of $S$ then we have that the following integrability condition holds
\begin{equation}\label{inte}
[\bar{\Lambda}^{\sharp}(\bar{\alpha}),\bar{\Lambda}^{\sharp}(\bar{\beta})]=
\bar{\Lambda}^{\sharp}\left({\mathcal{L}}_{\bar{\Lambda}^{\sharp}(\bar{\alpha})}
\bar{\beta}-{\mathcal{L}}_{\bar{\Lambda}^{\sharp}(\bar{\beta})}\bar{\alpha}-
\d(\bar{\beta}(\bar{\Lambda}^{\sharp}(\bar{\alpha})))\right).
\end{equation}
\end{itemize}
The triple $(M,S,\bar{\Lambda}^{\sharp})$ will be called {\rm a $k$-poly-Poisson manifold} or simply {\rm a poly-Poisson manifold}.

A $k$-poly-Poisson structure is said to be {\rm regular} if the vector bundle morphism
$\bar{\Lambda}^{\sharp}\colon S\subseteq (T^1_k)^*M\rightarrow TM$ has constant rank.
\end{definition}

\begin{lemma}\label{antisym}
Let $(S,\bar{\Lambda}^{\sharp})$ be $k$-poly-Poisson structure on $M$. Then,
\[
\bar{\alpha}(\bar{\Lambda}^{\sharp}(\bar{\beta}))=-\bar{\beta}(\bar{\Lambda}^{\sharp}(\bar{\alpha})),\qquad  \textrm{ for }\bar{\alpha}, \bar{\beta}\in S.
\]
\end{lemma}
\begin{pf}
By Definition \ref{def:poly-poisson} i) we have $0=(\bar{\alpha} + \bar{\beta}) (\bar{\Lambda}^{\sharp}(\bar{\alpha}+\bar{\beta}))=
\bar{\alpha}(\bar{\Lambda}^{\sharp}(\bar{\alpha}))+\bar{\alpha}(\bar{\Lambda}^{\sharp}(\bar{\beta})) +
\bar{\beta}(\bar{\Lambda}^{\sharp}(\bar{\alpha}))+ \bar{\beta}(\bar{\Lambda}^{\sharp}(\bar{\beta}))=
\bar{\alpha}(\bar{\Lambda}^{\sharp}(\bar{\beta})) + \bar{\beta}(\bar{\Lambda}^{\sharp}(\bar{\alpha}))$ and
this implies $\bar{\alpha}(\bar{\Lambda}^{\sharp}(\bar{\beta}))=-\bar{\beta}(\bar{\Lambda}^{\sharp}(\bar{\alpha}))$.
\end{pf}

As one would expect, polysymplectic manifolds are a particular example of $k$-poly-Poisson manifolds,
as it is shown in the following example.

\begin{example}[Polysymplectic manifolds] \label{ex:polysymplectic}
Let $(M,\bar{\omega})$ be a polysymplectic manifold. Since  $\bar{\omega}^{\flat}$ is a
monomorphism, $S:=\textrm{Im }(\bar{\omega}^{\flat})$ is a vector subbundle of $(T^1_k)^*M$.
Moreover, we consider
\[
\bar{\Lambda}^{\sharp}:=(\bar{\omega}^{\flat})^{-1}_{|S} \colon S\to TM.
\]
Clearly, i) in Definition \ref{def:poly-poisson} is equivalent to i) in Proposition
\ref{poly2}. On the other hand, if $\bar{\alpha}=\bar{\omega}^\flat(X)\in S$ satisfies
\[
\bar{\alpha}(\bar{\Lambda}^{\sharp}(\bar{\beta}))=0, \qquad \textrm{ for every }\bar{\beta}\in S,
\]
this implies that $\bar{\omega}^\flat(X)=0$. But, using that $\bar{\omega}^\flat$ is a monomorphism,
we deduce that $\bar{\Lambda}^{\sharp}(\bar{\alpha})=X=0$. Thus, ii) in Definition \ref{def:poly-poisson} holds.

Finally, the integrability condition \eqref{inte} is equivalent to
iii) in Proposition \ref{poly2}. Indeed, if $\bar{\alpha},\bar{\beta}\in \Gamma (S)$ then, using \eqref{eq:integrability:polysymplectic} and Lemma \ref{antisym}, it follows that
\[
\bar{\omega}^\flat [\bar{\Lambda}^{\sharp}(\bar{\alpha}),\bar{\Lambda}^{\sharp}(\bar{\beta})]=
{\mathcal{L}}_{\bar{\Lambda}^{\sharp}(\bar{\alpha})}
\bar{\beta}-{\mathcal{L}}_{\bar{\Lambda}^{\sharp}(\bar{\beta})}\bar{\alpha}-
\d(\bar{\beta}(\bar{\Lambda}^{\sharp}(\bar{\alpha}))).
\]
which implies that  \eqref{inte} holds.

It is not difficult to prove that polysymplectic structures are just poly-Poisson structures satisfying the non-degeneracy condition
\[
\textrm{Ker}\,(\bar{\Lambda}^{\sharp})=\{0\},
\]
which just says that $\bar{\Lambda}^{\sharp}$ is an isomorphism (note that, using Lemma \ref{antisym}, $\textrm{Ker}\,(\bar{\Lambda}^{\sharp})=(\textrm{Im }\bar{\Lambda}^\sharp)^\circ = \{
\bar{\alpha}\in (T^1_k)^*M\,|\, \bar{\alpha}(X)=0,\; \textrm{ for every } X\in \textrm{Im
}\bar{\Lambda}^\sharp \}$). \hfill$\diamond$

\end{example}

It is clear that Poisson manifolds are 1-poly-Poisson manifolds. Indeed, if $\Pi$ is a Poisson
structure on $M$ then $(S=T^*M,\Pi^\sharp)$ is poly-Poisson, where $\Pi^\sharp\colon T^*M\to TM$ is
the vector bundle morphism induced by the Poisson 2-vector $\Pi$. Moreover, it is well known that the generalized distribution
\[
p\in M\to \Pi^\sharp (T^*_pM)\subseteq T_pM
\]
is a symplectic generalized foliation (see, for instance, \cite{va}). Next, we will prove that a
$k$-poly-Poisson manifold admits a $k$-polysymplectic generalized foliation.

\begin{theorem} \label{thm:foliation}
Let $(S,\bar{\Lambda}^{\sharp})$ be a $k$-poly-Poisson structure on a manifold $M$. Then, $M$ admits a
$k$-polysymplectic generalized foliation $\mathcal{F}$ whose characteristic space at the point $p\in M$ is
\[
\mathcal{F}(p)=\bar{\Lambda}^{\sharp}(S_p)\subseteq T_pM.
\]
$\mathcal{F}$ is said to be the {\rm canonical $k$-polysymplectic foliation} of $M$.
\end{theorem}
\begin{pf}
Using that $S$ is a vector subbundle of $(T^1_k)^*M$, we can choose a local basis $\{\bar{\alpha}_1,\ldots, \bar{\alpha}_r\}$ of the space $\Gamma (S)$ of sections of $S$. Thus, it is clear that $\{ \bar{\Lambda}^{\sharp}(\bar{\alpha}_1),\ldots, \bar{\Lambda}^{\sharp}(\bar{\alpha}_r)\}$
is a local generator system of $\mathcal{F}$ which implies that $\mathcal{F}$ is locally finitely
generated. Integrability condition of
the poly-Poisson structure (Equation \eqref{inte}), makes this distribution involutive. Therefore,
$\mathcal{F}$ is a generalized foliation.

Next, let us prove that each leaf $\iota \colon L\to M$ of $\mathcal{F}$ is endowed with a $k$-polysymplectic structure $\bar{\omega}^\flat\colon TL\to (T^1_k)^*L$. Given $X\in T_pL
=\mathcal{F}(p)$ there exists $\bar{\alpha}\in S_p$ such that $\bar{\Lambda}^{\sharp}(\bar{\alpha})=X$. Then,
\begin{equation}\label{eq:induced:poly:symplectic}
\bar{\omega}^{\flat}(X)=\iota^*\bar{\alpha}.
\end{equation}
First, we will show that $\bar{\omega}^{\flat}$ is well defined. Assume that
$X=\bar{\Lambda}^{\sharp}(\bar{\alpha})=\bar{\Lambda}^{\sharp}(\bar{\beta})$.
Then $\bar{\Lambda}^{\sharp}(\bar{\alpha}-\bar{\beta})=0$ and, using Lemma \ref{antisym},
\begin{eqnarray*}
0&=& \bar{\gamma}(\bar{\Lambda}^{\sharp} (\bar{\alpha}-\bar{\beta}))\\
&=&-(\bar{\alpha}-\bar{\beta})(\bar{\Lambda}^{\sharp}(\bar{\gamma}))
\end{eqnarray*}
for any $\bar{\gamma} \in S$. As a consequence, $\iota^\ast(\bar{\alpha}-\bar{\beta})=0$.

It is easy to see that $\bar{\omega}^{\flat}$ is injective. Indeed, if
$X=\bar{\Lambda}^{\sharp}(\bar{\alpha})$ then $\bar{\omega}^{\flat}(X)=0$ is equivalent
to $\bar{\alpha}(\bar{\Lambda}^{\sharp}(\bar{\gamma}))=0$ for all $\bar{\gamma}\in S$.
But condition ii) of Definition \ref{def:poly-poisson}
implies that $X=\bar{\Lambda}^{\sharp}(\bar{\alpha})=0$.

To prove the integrability condition \eqref{eq:integrability:polysymplectic},
it is enough to prove the relation for a generating subset. Choose $\bar{\alpha}_1, \ldots, \bar{\alpha}_r$ a local
basis of $\Gamma(S)$ and denote $X_i=\bar{\Lambda}^{\sharp}(\bar{\alpha}_i)$,
$i=1,\ldots,r$, which generates the foliation $\mathcal{F}$.
Taking two elements $X_i=\bar{\Lambda}^{\sharp}(\bar{\alpha}_i)$, $X_j=\bar{\Lambda}^{\sharp}(\bar{\alpha}_j)$,
using \eqref{inte}, we have
\[
[X_i,X_j]=
\bar{\Lambda}^{\sharp}\left({\mathcal{L}}_{\bar{\Lambda}^{\sharp}(\bar{\alpha}_i)}\bar{\alpha}_j-
{\mathcal{L}}_{\bar{\Lambda}^{\sharp}(\bar{\alpha}_j)}\bar{\alpha}_i-
\d(\bar{\alpha}_j(\bar{\Lambda}^{\sharp}(\bar{\alpha}_i)))\right) .
\]
Applying $\bar{\omega}^\flat$ to both sides of the equation, from the fact that
$\bar{\omega}\circ\bar{\Lambda}^{\sharp}(\bar{\alpha})=\iota ^*\bar{\alpha}$, $\iota$ being
the inclusion from the leaf into the manifold $M$, this is just
\[
\bar{\omega}^{\flat}([X_i,X_j])=\iota ^*\left({\mathcal{L}}_{\bar{\Lambda}^{\sharp}
(\bar{\alpha}_i)}\bar{\alpha}_j-{\mathcal{L}}_{\bar{\Lambda}^{\sharp}(\bar{\alpha}_j)}
\bar{\alpha}_i-\d(\bar{\alpha}_j(\bar{\Lambda}^{\sharp}(\bar{\alpha}_i)))\right),
\]
and, from Equation \eqref{eq:induced:poly:symplectic} and the properties of Lie derivative and pull-back, we have
\[
\iota ^*\left({\mathcal{L}}_{\bar{\Lambda}^{\sharp}(\bar{\alpha}_i)}\bar{\alpha}_j-
{\mathcal{L}}_{\bar{\Lambda}^{\sharp}(\bar{\alpha}_j)}\bar{\alpha}_i-
\d(\bar{\alpha}_j(\bar{\Lambda}^{\sharp}(\bar{\alpha}_i)))\right)
={\mathcal{L}}_{X_i}\bar{\omega}^{\flat}(X_j)-{\mathcal{L}}_{X_j}\bar{\omega}^{\flat}(X_i)
-\d(\bar{\omega}^{\flat}(X_j)(X_i)),
\]
so the integrability condition for $\bar{\omega}$ holds.
\end{pf}

\begin{example}
Let $(M,\bar{\omega})$ be a connected polysymplectic manifold and consider the corresponding poly-Poisson manifold
$(M,S,\bar{\Lambda}^\sharp)$. Then, from Example \ref{ex:polysymplectic}, the distribution
$\bar{\Lambda}^{\sharp}(S)$ is just $TM$. Thus, there is only one leaf, the manifold $M$, and the
poly-symplectic structure on it is just $\bar{\omega}$.
\end{example}

Next, we will prove a converse of Theorem \ref{thm:foliation}.

Let $M$ be a smooth manifold, $S$ a vector subbundle of $(T^1_k)^*M$ and $\bar{\Lambda}^\sharp\colon S\to TM$ a
vector bundle morphism such that conditions i) and ii) in Definition \ref{def:poly-poisson} hold. If $p$ is a
point of $M$, we can define the nondegenerate $\rk$-valued 2-form $\bar{\omega}(p)\colon \bar{\Lambda}^\sharp
(S_p)\times \bar{\Lambda}^\sharp (S_p)\to \rk$ on the vector space $\bar{\Lambda}^\sharp (S_p)$ given by
\[
\bar{\omega}(p) ( \bar{\Lambda}^\sharp (\bar{\alpha}_p), \bar{\Lambda}^\sharp (\bar{\beta}_p))=\bar{\alpha}_p(\bar{\Lambda}^\sharp (\bar{\beta}_p)),\;\textrm{ for }\bar{\alpha}_p,\bar{\beta}_p\in S_p.
\]
In fact, since $\bar{\alpha} (\bar{\Lambda}^\sharp (\bar{\beta}))=-\bar{\beta} (\bar{\Lambda}^\sharp (\bar{\alpha}))$ for $\bar{\alpha},\bar{\beta}\in S$, we deduce 
that $\bar{\omega}(p)$ is well-defined.

Now, assume that the generalized distribution $\mathcal{F}$ on $M$ defined by
\[
p\in M\to \mathcal{F}(p)=\bar{\Lambda}^\sharp (S_p)\subseteq T_pM
\]
is involutive. Then, since $\mathcal{F}$ is locally finitely generated, we have that $\mathcal{F}$ is a generalized foliation.

In addition, if $L$ is a leaf of $\mathcal{F}$ then the restriction of $\bar{\omega}$ to $L$ induces a
nondegenerate $\rk$-valued 2-form $\bar{\omega}_L$ on $L$.  Therefore, if $\bar{\omega}_L$ is closed
we deduce that $\bar{\omega}_L$ defines a $k$-polysymplectic structure on $L$.

We will see that if this condition holds for every leaf $L$ then the couple $(S, \bar{\Lambda}^\sharp )$
is a $k$-poly-Poisson structure on $M$. Indeed, the result follows using Proposition \ref{poly2} (more
precisely Eq. \eqref{eq:integrability:polysymplectic}) and the fact that
\[
\bar{\omega}_L (\bar{\Lambda}^\sharp (\bar{\alpha})_{|L})=\iota^*\bar{\alpha},\qquad \textrm{ for } \bar{\alpha}\in \Gamma(S),
\]
where $\iota\colon L\to M$ is the canonical inclusion and we also denote by $\bar{\omega}_L$ the
vector bundle morphism between $TL$ and $(T^1_k)^*L$ induced by the $\rk$-valued 2-form
$\bar{\omega}_L$.

In conclusion, we have proved the following result.
\begin{theorem}\label{thm:inverse}
Let $S$ be a vector subbundle of $(T^1_k)^*M$ and $\bar{\Lambda}^\sharp \colon S\subseteq
{(T^1_k)^*M}\to TM$ a vector bundle morphism such that:
\begin{itemize}
\item[i)] $\bar{\alpha} (\bar{\Lambda}^{\sharp}(\bar{\alpha}))=0$, for $\bar{\alpha}
\in S$, and
\item[ii)] if $\bar{\alpha}(\bar{\Lambda}^{\sharp}(\bar{\beta}))=0$ for
every $\bar{\beta}\in S$ then $\bar{\Lambda}^{\sharp}(\bar{\alpha})=0$.
\end{itemize}
Assume also that the generalized distribution
\[
p\in M\to \mathcal{F}(p)=\bar{\Lambda}^\sharp (S_p)\subseteq T_pM
\]
is involutive and for every leaf $L$ the corresponding $\rk$-valued 2-form $\bar{\omega}_L$ on $L$ 
(induced by the morphism $\bar{\Lambda}^\sharp$) is closed. Then, the couple $(S,
\bar{\Lambda}^\sharp )$ is a $k$-poly-Poisson structure on $M$ and $\mathcal{F}$ is the canonical 
$k$-polysymplectic foliation of $M$.
\end{theorem}
Next, we will discuss a construction which allows to obtain some examples of $k$-poly-Poisson manifolds. This can be seen as a foliated version of the one presented in Example \ref{ex:regular:polysymplectic}. For this purpose, we will use Theorem \ref{thm:inverse}.

Let $M_A$ be a Poisson manifold, with $A\in\{1,\ldots ,k\}$, and $\pi_A\colon M\to M_A$
a fibration. Denote by $\Lambda ^A$ the Poisson 2-vector of $M_A$ and by $\mathcal{F}_A$ the symplectic foliation in $M_A$. Assume that $\mathcal{F}$ is a generalized foliation on $M$ such that
\begin{itemize}
\item[i)] For every $p\in M$
\begin{equation}\label{eq:compat:foliation}
(T_p\pi _A)(\mathcal{F}(p))=\mathcal{F}_A(\pi_A(p)).
\end{equation}
\item[ii)] For every $p\in M$
\begin{equation}\label{eq:compat:forms}
\big ( \displaystyle \cap_{A=1}^k {\rm Ker\,}(T_p\pi _A)\big )\cap \mathcal{F}(p)=\{0\}.
\end{equation}
\end{itemize}
We will show that every leaf $L$ of $\mathcal{F}$ admits a $k$-polysymplectic structure. We will proceed
as in Section \ref{poly_manifolds}.

Let $p$ be a point of $M$, $L$ the leaf of $\mathcal{F}$ over the point $p$ and $L_A$ the symplectic leaf of
$M_A$ over the point $\pi _A(p)$, for every $A\in\{1,\ldots ,k\}$. First of all, we will define a
$k$-polysymplectic structure $(\omega ^1_L(p),\ldots ,\omega ^k_L(p))$ on the vector space
$\mathcal{F}(p)=T_pL$. In fact, if $\Omega _{L_A}(\pi _A(p))$ is the symplectic 2-form on the vector space
$\mathcal{F}_A(\pi _A(p))=T_{\pi _A(p)}L_A$ then
\begin{equation}\label{eq:forma:hoja}
\omega^A_L(p)=(T_p^*\pi _A)(\Omega _{L_A}(\pi _A(p))).
\end{equation}
\begin{remark}
If $u,v\in \mathcal{F}(p)=T_pL$, we have that
\[
(T_p\pi _A)(u)=\Lambda^\sharp_A(\widetilde{\alpha}_A),\; (T_p\pi _A)(v)=\Lambda^\sharp_A(\widetilde{\beta}_A), \qquad \mbox{ with }\widetilde{\alpha}_A,\widetilde{\beta}_A
\in T^*_{\pi_A(p)}M_A.
\]
Thus, since
\[
\Omega_{L_A}(\pi _A(p))(\Lambda _A^\sharp (\widetilde{\alpha}_A), \Lambda _A^\sharp (\widetilde{\beta}_A) )=\widetilde{\alpha}_A(\Lambda _A^\sharp (\widetilde{\beta}_A)),
\]
we obtain that
\begin{equation}\label{eq:forma:poly:hoja}
\omega ^A_L(u,v)=\widetilde{\alpha}_A(\Lambda _A^\sharp (\widetilde{\beta}_A)).
\end{equation}
\hfill$\diamond$
\end{remark}
From \eqref{eq:compat:foliation}, \eqref{eq:compat:forms} and \eqref{eq:forma:hoja}, it follows that
$(\omega ^1_L (p),\ldots ,\omega ^k_L(p) )$ is a $k$-polysymplectic structure on the vector space
$\mathcal{F}(p)=T_pL$, that is,
\begin{equation}\label{3.3v}
\cap _{A=1}^k\textrm{Ker }(\omega _L^A(p))=\{0 \}.
\end{equation}
Therefore, $L$ admits an almost $k$-polysymplectic structure $(\omega ^1_L,\ldots ,\omega ^k_L)$, i.e.,
$\omega ^A_L$ is a 2-form on $L$, for every $A\in\{1,\ldots ,k\}$, and
\[
\cap _{A=1}^k\textrm{Ker }\omega _L^A=\{0 \}.
\]
In addition, using \eqref{eq:compat:foliation}, \eqref{eq:compat:forms} and the fact that $\Omega_{L_A}$
is a closed 2-form on $L_A$, we deduce that the 2-forms $\omega_L^A$ are also closed and $(\omega^1_L,\ldots ,\omega ^k_L)$ is a $k$-polysymplectic structure on $L$. Consequently, $\mathcal{F}$ is a $k$-polysymplectic foliation.

Now, we will introduce a vector subspace $S(p)$ of $(T^1_k)_p^*M$ and a linear epimorphism
$\bar{\Lambda}^\sharp\colon S(p)\subseteq (T^1_k)^*M\to \mathcal{F}(p)\subseteq T_pM$.

From \eqref{eq:compat:foliation} and \eqref{eq:compat:forms}, we have that the linear map
$T_p\pi\colon \mathcal{F}(p)\to \mathcal{F}_1(\pi_1(p))\times \ldots \times \mathcal{F}_k(\pi_k(p))$ given by
\[
(T_p\pi )(u)=((T_p\pi _1)(u),\ldots ,(T_p\pi _k)(u)),\qquad \textrm{ for }u\in \mathcal{F}(p),
\]
is a linear monomorphism.

Next, we define
\begin{equation}\label{def:S}
\begin{array}{rcr}
S(p)&=&\{ ((T^*_p\pi _1)(\widetilde{\alpha}_1),\ldots ,(T^*_p\pi _k)(\widetilde{\alpha}_k))\in
(T^1_k)_p^*M \,|\, \exists u\in \mathcal{F}(p) \textrm{ and } \\ &&(T_p\pi)(u)=(\Lambda _1^\sharp (\widetilde{\alpha}_1),\ldots ,\Lambda _k^\sharp (\widetilde{\alpha}_k))\}
\end{array}
\end{equation}
and the linear epimorphism $\bar{\Lambda}^\sharp\colon S(p)\subseteq (T^1_k)_p^*M\to \mathcal{F}(p)$
by
\[
\bar{\Lambda}^\sharp  ((T^*_p\pi _1)(\widetilde{\alpha}_1),\ldots ,(T^*_p\pi _k)(\widetilde{\alpha}_k))=u,
\; \textrm{ if } (T_p\pi)(u)=(\Lambda _1^\sharp (\widetilde{\alpha}_1),\ldots ,\Lambda _k^\sharp (\widetilde{\alpha}_k)).
\]
Note that $\bar{\Lambda}^\sharp$ is well-defined because $T_p\pi$ is injective. Moreover, we will
see that the couple $(S(p),\bar{\Lambda}^\sharp )$ satisfies conditions i) and ii) in Theorem \ref{thm:inverse}.

First, if
\[
\bar{\alpha}=(\alpha _1,\ldots ,\alpha _k)=((T^*_p\pi _1)(\widetilde{\alpha}_1),\ldots ,(T^*_p\pi _k)(\widetilde{\alpha}_k))\in S(p)
\]
it follows that
\[
\alpha _A(\bar{\Lambda}^\sharp (\bar{\alpha}))=\widetilde{\alpha}_A((T_p\pi _A)(u))=\widetilde{\alpha}_A(\Lambda _A^\sharp (\widetilde{\alpha}_A))=0,\;\textrm{ for all }A.
\]
Furthermore, if for every $\bar{\beta}\in S(p)$
\[
\bar{\alpha}(\bar{\Lambda}^\sharp (\bar{\beta}))=0
\]
then we will prove that $\bar{\Lambda}^\sharp (\bar{\alpha})=0$. Indeed, suppose that $u=\bar{\Lambda}^\sharp
(\bar{\alpha})\in\mathcal{F}(p)$. This implies that $(T_p\pi_A)(u)=\Lambda^\sharp _A(\widetilde{\alpha}_A)$.
Now, let $v$ be an arbitrary element of $\mathcal{F}(p)$. We have that
\[
(T_p\pi _A)(v)=\Lambda_A^\sharp (\widetilde{\beta}_A),\qquad \textrm{ with }\widetilde{\beta}_A
\in T^*_{\pi_A(p)}M_A,
\]
for every $A\in \{1,\ldots ,k\}$. Thus,
\[
\bar{\beta}=((T^*_p\pi_1)(\widetilde{\beta}_1),\ldots ,(T^*_p\pi_k)(\widetilde{\beta}_k))\in S(p)
\]
and
\[
0=\alpha _A(\bar{\Lambda}^\sharp (\bar{\beta}))=\widetilde{\alpha}_A(\Lambda_A^\sharp (\widetilde{\beta}_A)), \qquad \textrm{ for every }A\in \{1,\ldots ,k\}.
\]
Therefore, using \eqref{eq:forma:poly:hoja}, it is deduced that
\[
0= \omega_L^A(p)(u,v), \qquad \textrm{ for every } A\in \{1,\ldots ,k\}.
\]
Consequently, from \eqref{3.3v}, we obtain that
\[
0=u=\bar{\Lambda}^\sharp (\bar{\alpha}).
\]
So, if the assignment
\[
p\in M\mapsto S(p)\subseteq (T^1_k)^*_pM
\]
defines a vector subbundle  of  $(T^1_k)^*M$ then, using Theorem \ref{thm:inverse}, it follows that the couple $(S,\bar{\Lambda}^\sharp )$ is a $k$-poly-Poisson structure on $M$ with canonical
$k$-poly-symplectic foliation $\mathcal{F}$.

We will see that this construction works in some interesting examples.

\begin{example}[The product of Poisson manifolds]
 Let $(M_A,\Lambda _A)$ be a Poisson manifold, with $A\in\{1,\ldots ,k\}$, and $M$ the product
manifold $M=M_1\times\ldots\times M_k$. Denote by $\pi _A\colon M\to M_A$ the canonical projection
and by $\mathcal{F}_A$ the symplectic foliation in $M_A$.

We consider the foliation $\mathcal{F}$ in $M$ whose characteristic space at the point $p=(p_1,\ldots ,p_k)
\in M_1\times \ldots \times M_k=M$ is
\[
\mathcal{F}(p)=\mathcal{F}(p_1,\ldots ,p_k)=\mathcal{F}_1(p_1)\times \ldots \times \mathcal{F}_k(p_k)
\subseteq T_{p_1}M_1\times \ldots \times T_{p_k}M_k\cong T_pM.
\]
It is clear that  \eqref{eq:compat:foliation} and \eqref{eq:compat:forms} hold. Thus, the leaves of
$\mathcal{F}$ admit a $k$-polysymplectic structure. If $p=(p_1,\ldots ,p_k)\in M_1\times \ldots \times M_k=M$
then, in this case, the map
\[
T_p\pi\colon \mathcal{F}(p)\to \mathcal{F}_1(p_1)\times \ldots \times \mathcal{F}_k(p_k)
\]
is a linear isomorphism. This implies that
\[
S(p)=\{ ((T_p^*\pi _1)(\widetilde{\alpha}_1),\ldots ,(T_p^*\pi _k)(\widetilde{\alpha}_k))\in (T^1_k)^*_pM
\,|\, \widetilde{\alpha}_A\in T^*_{\pi _A(p)}M_A,\textrm{ for all }A\}.
\]
In other words,
\[
S(p)\cong T^*_{\pi _1(p)}M_1\times \ldots \times  T^*_{\pi _k(p)}M_k\cong T^*_pM.
\]
Therefore, the assignment
\[
p\in M\mapsto S(p)\subseteq (T^1_k)^*_pM
\]
is a vector subbundle of $(T^1_k)^*_pM$ of rank $m$, with $m=\sum _{A=1}^km_A$ and $m_A=$ dim $M_A$, which may
be identified with the cotangent bundle $T^*M\to M$ to $M$.

Under this identification, the vector bundle morphism
\[
\bar{\Lambda}^\sharp \colon S\cong T^*M=T^*M_1\times \ldots\times T^*M_k\to
\mathcal{F}\subseteq TM=TM_1\times \ldots TM_k
\]
is given by
\[
\bar{\Lambda}^\sharp (\widetilde{\alpha}_1,\ldots ,\widetilde{\alpha}_k)=(\Lambda_1^\sharp
(\widetilde{\alpha}_1),\ldots ,\Lambda_k^\sharp (\widetilde{\alpha}_k)),
\]
for $(\widetilde{\alpha}_1,\ldots ,\widetilde{\alpha}_k)\in T^*M_1\times \ldots \times T^*M_k$.
\hfill$\diamond$
\end{example}

\begin{example}[The Whitney sum $E^*_k=E^*\oplus \stackrel{(k}{\ldots}\oplus E^*$, with $E$
a Lie algebroid]\label{ex:whitney-algebroid} Let $\tau _E\colon E\to Q$ be a Lie algebroid of rank $n$ over a manifold $Q$ of dimension $m$, $\Lambda_ {E^*}$ the fiberwise linear Poisson structure on the
dual bundle $E^*$ to $E$ and $\mathcal{F}_{E^*}$ the symplectic foliation associated with $\Lambda_{E^*}$ (see Appendix).

We consider the Whitney sum $E^*_k=E^*\oplus \ldots \oplus E^*$. It is a vector bundle over $Q$ whose fiber at
the point $q\in Q$ is $(E^*_k)_q=E_q^*\times \ldots \times E_q^*$. We will denote by $\tau _{E^*_k}\colon
E^*_k\to Q$ the vector bundle projection.

If $\bar{\alpha}=(\alpha _1,\ldots ,\alpha _k)\in (E^*_k)_q$ we have that
\[
T_{\bar{\alpha}}E^*_k=\{ (\widetilde{v}_1,\ldots ,\widetilde{v}_k)\in T_{\alpha_1}E^*\times \ldots \times
T_{\alpha _k}E^*\,|\, (T_{\alpha _A}\tau _{E^*})(\widetilde{v}_A)=(T_{\alpha _B}\tau _{E^*})(\widetilde{v}_B),
\textrm{ for all }A\textrm{ and }B\},
\]
where $\tau_{E^*}\colon E^*\to Q$ is the vector bundle projection. Thus, if $\widetilde{\alpha}_A
\in T^*_{\alpha_A}E^*$ for every $A$ then, using \eqref{eq:relation:lin:Poisson} (see Appendix), it follows that
\[
(\Lambda_{E^*}^\sharp (\widetilde{\alpha}_1),\ldots ,\Lambda_{E^*}^\sharp (\widetilde{\alpha}_k))
\in T_{\bar{\alpha}}E^*_k
\]
if and only if
\[
\widetilde{\alpha}_A\smalcirc \, ^{\textrm{v}}_{\alpha_A}\smalcirc
\rho^*_{E}{}_{|T^*_qQ}=\widetilde{\alpha}_B\smalcirc \, ^{\textrm{v}}_{\alpha_B}\smalcirc
\rho^*_{E}{}_{|T^*_qQ}, \qquad \textrm{ for all }A\textrm{ and }B,
\]
$\rho_E\colon E\to TQ$ being the anchor map of $E$, $\rho ^*_E\colon T^*Q\to E^*$ the dual
morphism of $\rho _E$ and $^{\textrm{v}}_{\alpha_C}\colon E^*_q\to T_{\alpha _C}E^*_q$ the
canonical isomorphism between $E^*_q$ and $T_{\alpha _C}E^*_q$, for every $C\in \{1,\ldots ,k\}$.

Denote by $\pi_A\colon E^*_k\to E^*$ the canonical projection given by
\[
\pi _A(\alpha_1,\ldots ,\alpha _k)=\alpha _A,\qquad \textrm{ for }(\alpha _1,\ldots ,\alpha _k)\in E^*_k.
\]
Then, it is clear that $\pi_A\colon E^*_k\to E^*$  is a surjective submersion. In fact,
\begin{equation}\label{eq:Tangent-PiA}
(T_{\bar{\alpha}}\pi _A)(\widetilde{v}_1,\ldots ,\widetilde{v}_k)=\widetilde{v}_A,
\qquad \textrm{ for } (\widetilde{v}_1,\ldots ,\widetilde{v}_k)\in T_{\bar{\alpha}}E^*_k.
\end{equation}
Now, let $\mathcal{F}$ be the generalized distribution on $E^*_k$ whose characteristic space at the point
$\bar{\alpha}=(\alpha _1,\ldots ,\alpha _k)\in (E^*_k)_q$ is
\begin{equation}\label{eq:Def-Foliation}
\begin{array}{rcr}
\mathcal{F}(\bar{\alpha})&=&\{ (\Lambda^\sharp _{E^*}(\widetilde{\alpha}_1),\ldots ,\Lambda^\sharp
_{E^*}(\widetilde{\alpha}_k))\in T_{\alpha _1}E^*\times \ldots \times T_{\alpha _k}E^* \,|\,
\widetilde{\alpha}_A\in T^*_{\alpha _A}E^* \textrm{ and } \\[8pt]
&& \widetilde{\alpha}_A\smalcirc \,^{\textrm{v}}_{\pi_A(\bar{\alpha})}=\widetilde{\alpha}_B\smalcirc
\,^{\textrm{v}}_{\pi_B(\bar{\alpha})}\}
\end{array}
\end{equation}
Note that the condition $\widetilde{\alpha}_A\smalcirc \,^{\textrm{v}}_{\pi_A(\bar{\alpha})}
=\widetilde{\alpha}_B\smalcirc \,^{\textrm{v}}_{\pi_B(\bar{\alpha})}$, for all $A$ and $B$, implies that
\[
(\Lambda^\sharp _{E^*}(\widetilde{\alpha}_1),\ldots ,\Lambda^\sharp _{E^*}(\widetilde{\alpha}_k)) \in
T_{\bar{\alpha}}E^*_k.
\]
Next, we will prove that $\mathcal{F}$ is a generalized foliation on $E^*_k$. First of all, we will see that
$\mathcal{F}$ is locally finitely generated.

Let $(q^i)$ be local coordinates on an open subset $U\subseteq Q$ and $\{e_\alpha\}$ be a local basis
of $\Gamma (\tau_E^{-1}(U))$. Denote by $(q^i,p_\alpha )$ the corresponding local coordinates on
$E^*$. Then, we have local coordinates $(q^i,p^A_\alpha )$ on $E^*_k=E^*\oplus \ldots \oplus E^*$.
In fact, if $\bar{\alpha}\in \tau_{E^*_k}^{-1}(U)$
\[
q^i(\overline{\alpha })=q^i(\tau_{E^*_k}(\bar{\alpha})),\qquad
p^A_\alpha (\overline{\alpha })=p_\alpha (\pi _A (\bar{\alpha})),
\]
for $i\in \{1,\ldots , m\}$, $\alpha \in\{1,\ldots ,n\}$ and $A\in \{1,\ldots ,k\}$.

Moreover,
\begin{align*}
\{ (0,\ldots, \stackrel{A)}{\Lambda^\sharp_{E^*}(\d q^i)},\ldots ,0), (\Lambda^\sharp_{E^*}(\d p_\alpha),\ldots , \Lambda^\sharp_{E^*}(\d p_\alpha ))&\;\; | &i\in \{1,\ldots , m\},\, \alpha \in\{1,\ldots ,n\} \\
& &\textrm{  and } A\in \{1,\ldots ,k\}\}
\end{align*}
is a local generator system of $\mathcal{F}$. Thus, using \eqref{eq:linear:Poisson}, it follows that
\[
\{ \rho ^i_\beta \frac{\partial }{\partial p^A_\beta},  \rho ^i_\alpha \frac{\partial }{\partial q^i} +{\mathcal C}_{\alpha \beta}^\gamma p^B_\gamma \frac{\partial}{\partial p^B_\beta}  \,|\,
i\in \{1,\ldots , m\},\, \alpha \in\{1,\ldots ,n\}
\textrm{  and } A\in \{1,\ldots ,k\}\}
\]
is a local generator system of $\mathcal{F}$, where $(\rho^i_\beta ,{\mathcal C}_{\alpha \beta}^\gamma)$ are the local structure functions of $E$. Therefore, $\mathcal{F}$ is locally
finitely generated.

On the other hand, from \eqref{estruc1}, we deduce that
\[
\begin{array}{rcl}
\left [ \rho ^i_\alpha \frac{\partial }{\partial p^A_\alpha },\rho ^j_\beta \frac{\partial }{\partial p^B_\beta} \right ] &=&0\\
\left [ \rho ^i_\alpha  \frac{\partial }{\partial p^A_\alpha},  \rho ^j_\beta \frac{\partial }{\partial q^j} +
{\mathcal C}_{\beta\gamma }^\mu p^B_\mu \frac{\partial}{\partial p^B _\gamma } \right ] &=&-\frac{\partial \rho ^i_\beta }{\partial q^j}(\rho ^j_\gamma \frac{\partial}{\partial p^A_\gamma} )\\
\left [  \rho ^i_\alpha \frac{\partial }{\partial q^i} +
{\mathcal C}_{\alpha \beta}^\gamma p^A_\gamma \frac{\partial}{\partial p^A_\beta} ,\rho ^j_\mu \frac{\partial }{\partial q^j} + {\mathcal C}_{\mu\nu }^\theta p^B_\theta \frac{\partial}{\partial p^B_\nu}\right ] &=&   {\mathcal C}_{\alpha \mu}^\nu (\rho ^i_\nu \frac{\partial }{\partial q^i}-{\mathcal C}_{\beta\nu}^\theta
p^A_{\theta}\frac{\partial}{\partial p^A_\beta })-p^A_\theta \frac{\partial {\mathcal C}^\theta _{\alpha \mu}}{\partial q^i}(\rho ^i_\beta \frac{\partial}{\partial p^A_\beta}).
\end{array}
\]
This implies that $\mathcal{F}$ is involutive. Consequently, $\mathcal{F}$ is a generalized foliation on
$E^*_k$.

Now, we will prove that conditions \eqref{eq:compat:foliation} and \eqref{eq:compat:forms} hold
for the foliation $\mathcal{F}$. We have that
\[
\mathcal{F}_{E^*}(\alpha )=\Lambda ^\sharp _{E^*}(T^*_\alpha E^*),\qquad \textrm{ for }\alpha \in E^*.
\]
Thus, from \eqref{eq:Tangent-PiA} and \eqref{eq:Def-Foliation}, it follows that
\[
(T_{\bar{\alpha}}\pi _A)(\mathcal{F}(\bar{\alpha}))\subseteq \mathcal{F}_{E^*}(\pi _A(\bar{\alpha})),\qquad \textrm{ for }\bar{\alpha}\in E^*_k\textrm{ and } A\in\{1,\ldots ,k\}.
\]
Moreover, if $A\in\{1,\ldots ,k\}$ and $\widetilde{v}_A\in \mathcal{F}_{E^*}(\pi _A(\bar{\alpha}))$ then
\[
\widetilde{v}_A=\Lambda_{E^*}^ \sharp (\widetilde{\alpha}_A),\qquad \textrm{ with }\widetilde{\alpha}_A
\in T^*_{\pi _A(\bar{\alpha})}E^*.
\]
In addition, it is easy to prove that we can choose $\widetilde{\alpha}_B\in T^*_{\pi _B(\bar{\alpha})}E^*$, with $B\in\{1,\ldots ,k\}$, satisfying
\[
\widetilde{\alpha}_B\smalcirc \,^{\textrm{v}}_{\pi_B(\bar{\alpha})}=
\widetilde{\alpha}_A\smalcirc \,^{\textrm{v}}_{\pi_A(\bar{\alpha})}.
\]
Therefore, $(\Lambda _{E^*}^\sharp (\widetilde{\alpha}_1),\ldots ,\Lambda _{E^*}^\sharp (\widetilde{\alpha}_k))\in \mathcal{F}(\bar{\alpha})$ and, using \eqref{eq:Tangent-PiA}, we deduce that
\[
(T_{\bar{\alpha}}\pi_A)(\Lambda _{E^*}^\sharp (\widetilde{\alpha}_1),\ldots ,\Lambda _{E^*}^\sharp (\widetilde{\alpha}_k))=\Lambda_{E^*}^\sharp (\widetilde{\alpha}_A)=\widetilde{v}_A.
\]
This proves \eqref{eq:compat:foliation}  for the foliation $\mathcal{F}$.

On the other hand, using again \eqref{eq:Tangent-PiA}, we obtain that
\[
\displaystyle \cap _{A=1}^k \textrm{Ker }(T_{\bar{\alpha}}\pi _A)=\{0\}
\]
and, as a consequence, \eqref{eq:compat:forms} also holds for the foliation $\mathcal{F}$.

Next, we will see that the assignment
\[
\bar{\alpha}\in E^*_k\to S(\bar{\alpha})\subseteq (T^1_k)^*_{\bar{\alpha}}E_k^*
\]
defines a vector subbundle of $(T^1_k)^*E_k^*$.

In fact, if $\bar{\alpha}\in E^*_k$ then, from  \eqref{def:S}, we have that
\begin{equation}\label{eq:dist:whitney}
\begin{array}{rcr}
S(\bar{\alpha})&=&\{ ((T^*_{\bar{\alpha}}\pi _1)(\widetilde{\alpha}_1),\ldots ,(T^*_{\bar{\alpha}}\pi _k)(\widetilde{\alpha}_k))\in
(T^1_k)_{\bar{\alpha}}^*E^*_k  \,|\, \widetilde{\alpha}_A\in T^*_{\pi _A(\bar{\alpha})}E^*
\textrm{ and } \\ &&   \widetilde{\alpha}_A\smalcirc \,^{\textrm{v}}_{\pi_A(\bar{\alpha})}=
\widetilde{\alpha}_B\smalcirc \,^{\textrm{v}}_{\pi_B(\bar{\alpha})} \textrm{ for all $A$ and $B$}\}
\end{array}
\end{equation}
This implies that the assignment
\[
\bar{\alpha}\in E^*_k\to S(\bar{\alpha})\subseteq (T^1_k)^*_{\bar{\alpha}}E_k^*
\]
defines a vector subbundle of $(T^1_k)^*E^*_k$ of rank $mk+n$.

Moreover, in this case, the linear epimorphism
\[
\bar{\Lambda}^\sharp\colon S(\bar{\alpha})\subseteq (T^1_k)_{\bar{\alpha}}^*E^*_k
\to \mathcal{F}(\bar{\alpha})\subseteq T_{\bar{\alpha}}E^*_k\subseteq T_{\pi_1(\bar{\alpha})}E^*\times \ldots \times T_{\pi_k(\bar{\alpha})}E^*
\]
is given by
\[
\bar{\Lambda}^\sharp ((T^*_{\bar{\alpha}}\pi _1)(\widetilde{\alpha}_1),\ldots ,(T^*_{\bar{\alpha}}\pi _k)(\widetilde{\alpha}_k))=
(\Lambda_{E^*}^\sharp (\widetilde{\alpha}_1),\ldots ,\Lambda_{E^*}^\sharp (\widetilde{\alpha}_k)).
\]
From the above results, it can be concluded that the couple $(S,\bar{\Lambda}^\sharp)$ is a $k$-poly-Poisson
structure on $E^*_k$.\hfill$\diamond$

\end{example}

\begin{remark}\label{rmk:whitney}
If we consider the particular case when $E$ is the standard Lie algebroid $TQ$ then it is clear that $E^*_k$ is
the cotangent bundle of $k$-covelocities $\tkqh$. In addition, a direct computation proves that the couple
$(S,\bar{\Lambda}^\sharp)$ is just the $k$-poly-Poisson structure induced by the $k$-polysymplectic structure on
$\tkqh$ which was described in Example \ref{canexample}. Note that the fiberwise linear Poisson structure on
$T^*Q$ induced by the standard Lie algebroid structure on $TQ$ is just the canonical Poisson structure
associated to the canonical symplectic structure on $T^*Q$ (see Appendix).

On the other hand, if $E$ is a Lie algebra $\mathfrak{g}$, then $E^*_k$ is just $k$ copies of $\mathfrak{g}^*$,
the dual of the  Lie algebra, which will be denoted by $\mathfrak{g}^*_k$. Moreover, using the trivialization
$T^*\mathfrak{g}^*\cong \mathfrak{g}\times\mathfrak{g}^*$, it is deduced that $(T^1_k)^*\mathfrak{g}^*_k\cong
\mathfrak{g}^*_k\times (\mathfrak{g}_k\times \stackrel{(k}{\ldots}\times \mathfrak{g}_k)$, $\mathfrak{g}_k$
being $k$ copies of $\mathfrak{g}$. Under this identification, given a point $\mu =(\mu _1,\ldots ,\mu _k)\in
\mathfrak{g}^*_k$, $S$ is defined by
\begin{equation}\label{eq:algebra:S}
S(\mu )=\{ (\mu, ((\xi,0,\ldots,0),(0,\xi,\dots,0),\ldots,(0,0,\ldots,\xi))) \, |\, \xi\in\mathfrak{g}\}
\end{equation}
and $\bar{\Lambda}^\sharp$ is characterized by the following expression
\begin{eqnarray}\label{eq:algebra:lambda}
\bar{\Lambda}^{\sharp}_{\mu}(\mu, ((\xi,0,\ldots,0),(0,\xi,\dots,0),\ldots,(0,0,\ldots,\xi)))
&=&(ad^*_{\xi}\mu_1,ad^*_{\xi}\mu_2,\ldots,ad^*_{\xi}\mu_k)\\
&=&(\xi_{\mathfrak{g}^*}(\mu _1),\ldots ,\xi_{\mathfrak{g}^*}(\mu _k)).
\end{eqnarray}
Therefore, from \eqref{eq:inf:k:coadjoint}, we have that each polysymplectic leaf is a $k$-coadjoint orbit of
Example \ref{ex:k-coadjoint orbit}.

Another interesting example, which generalizes the previous one, is the Atiyah algebroid of
a $G$-principal bundle $Q\to Q/G=M$. In this case, the Lie algebroid is the vector bundle $TQ/G\to M$.
In section \ref{sec:reduction}, we will describe the poly-Poisson structure on $T^*Q/G\oplus \stackrel{(k}{\ldots} \oplus T^*Q/G$ as a consequence of a reduction procedure.\hfill$\diamond$
\end{remark}
\begin{example}\label{ex:poly:frame}
An example closely related with Example \ref{ex:whitney-algebroid} is the following one. Given
a vector bundle $\tau _E\colon E\to Q$ of rank $n$ over a manifold $Q$ of dimension $m$, its frame bundle
is the bundle $LE\to M$ whose fiber at $q\in Q$ is the set of linear isomorphisms at $E_q$, i.e., $u\in (LE)_q$
if $u\colon \R^n \to E_q$ is a linear isomorphism.

The frame bundle is a $GL(n)$-principal bundle. More precisely, if $u\in LE$ and $g\in GL(n)$ then
$u\cdot g$ is just the composition. Moreover, we have that $LE$ is an open subset of the Whitney sum
$E^*_n=E^*\oplus \stackrel{(n}{\ldots} \oplus E^*$.

Now, if $\tau _E\colon E\to Q$ is a Lie algebroid with $\Lambda_ {E^*}$ the fiberwise linear Poisson structure on the dual bundle $E^*$ to $E$, then $LE$ is endowed with a poly-Poisson structure. Consider the maps $\pi _A\colon LE\to E^*$ given by
\[
u\in LE \mapsto \pi_A(u)=\textrm{pr}_A\circ u^{-1}.
\]
Note that these maps are just the restriction to $LE$ of the canonical projections $\pi _A\colon E^*_n\to E^*$, $A\in \{1,\ldots ,n\}$, of Example \ref{ex:whitney-algebroid}. Thus, taking the foliation $\mathcal{F}$ of the same example restricted to $LE$, one immediately deduce that $LE$ is endowed with a poly-Poisson structure.\hfill$\diamond$

\end{example}
Next, it will be considered a different type of example of a poly-Poisson structure, motivated by singular Lagrangian systems.
\begin{example}[$k$-poly-Poisson structures of Dirac type]\label{ex:Dirac:type}

One of the typical examples of Poisson structures is given by the so-called Poisson structures of Dirac type,
which are useful for dealing with singular Lagrangian systems. Let $(M,\omega)$ be a symplectic manifold and let $D$ denote a
regular involutive distribution on $M$ such that $D\cap D^{\perp}=0$ where $\perp$ is used to denote the symplectic
orthogonal (recall that given a vector subspace $W\subseteq T_pM$, $W^{\perp}= (\omega^{\flat})^{-1}(W^{\circ})$
where $^{\circ}$ denotes the annihilator). Then, one may define a Poisson structure $\Lambda_D$ on $M$ just by
describing the associated bundle map $\Lambda_D^\sharp$ in the following way:
\[
\Lambda_D^{\sharp}(\alpha)= p\circ(\omega^{\flat})^{-1}(\alpha),  \qquad (\alpha \in T^*M),
\]
$p:TM\rightarrow D$ being the projection over $D$ (note that $TM=D\oplus D^\perp$). It
can be seen that $\Lambda_D^{\sharp}$ defined in this way is a Poisson structure. Furthermore, the associated symplectic foliation is just $D$ (for more details see,
for instance, \cite[3.8]{va}).

We are going to extend this construction to the poly-Poisson context.
Let $(\omega_1,\ldots,\omega_k)$ be a $k$-polysymplectic structure on a manifold $M$ of dimension $m$ and $D$ be a
regular and involutive distribution of rank $r$ on $M$. In addition, assume that $D$ satisfies
\begin{equation}\label{eq:symplectic:sub}
D(p)\cap D^{\perp}(p)=\{0\},\qquad \textrm{ for every }p\in M,
\end{equation}
where, in  this case, $D^{\perp}$ is defined by
\[
D^{\perp}(p):=\cap_{i=1}^k({\omega_i}^{\flat})^{-1}( D^{\circ}(p) ).
\]
First, define $S$ as the vector subbundle of $(T^1_k)^*M$ whose fiber at $p\in M$ is
\[
S(p)=\{\bar{\alpha}=(\alpha_1,\ldots,\alpha_k)\in (T^1_k)_p^*M \, |\,
\exists X\in D(p) \textrm{ such that } \bar{\alpha}_{|D(p)}=\bar{\omega}^{\flat}(X)_{|D(p)}\}.
\]
Note that, if $\iota _D(p)\colon D(p)\to T_pM$ is the canonical inclusion, then the dimension of the space $\bar{S}(p)=[(\iota _D(p)^*\times \ldots \times \iota _D(p)^*)\smalcirc \bar{\omega}^\flat(p)](D(p))$ is just $r=\mbox{dim }D(p)=\mbox{rank }D$ (this follows using
\eqref{eq:symplectic:sub}). Thus, since
\[
S(p)=(\iota _D(p)^*\times \ldots \times \iota _D(p)^*)^{-1}(\bar{S}(p))
\]
we conclude that the rank of $S$ is $r+(m-r)k$.

Next, the morphism $\bar{\Lambda}^{\sharp}\colon S\to TM$ is defined as follows: If $\bar{\alpha}\in S$ then there exists $X\in D$ such that $\bar{\omega}^\flat(X)_{|D}=\bar{\alpha}_{|D}$. Thus,
\[
\bar{\Lambda}^{\sharp}(\bar{\alpha})=X.
\]
$\bar{\Lambda}^\sharp$ is well defined: if $X,Y\in D$ and $\bar{\omega}^\flat(X)_{|D}=\bar{\omega}^\flat(Y)_{|D}$, then $X-Y\in D^\perp$. Therefore, using that $D\cap D^\perp=\{ 0\}$, we get $X=Y$.
Let us show that the conditions of Theorem \ref{thm:inverse} hold.

If $\bar{\alpha}\in S$ then there exists $X\in D$ such that $\bar{\omega}^{\flat}(X)_{|D} =\bar{\alpha}_{|D}$.
Thus,
\[
\bar{\alpha}(\bar{\Lambda}^\sharp (\bar{\alpha}))=\bar{\alpha}(X)=\bar{\omega}^\flat (X)(X)=0,
\]
where, in the last equality, it has been used i) in Proposition \ref{poly2}.

Next, assume that $\bar{\alpha}\in S$ satisfies that  $\bar{\alpha}(\bar{\Lambda}^{\sharp}(\bar{\beta}))=0$ for
every $\bar{\beta}\in S$. As a consequence,
\begin{equation}\label{3.12}
\bar{\alpha}(Y)=0,
\end{equation}
for every $Y\in D$. But, if $\bar{\alpha}_{|D}=\bar{\omega}^{\flat}(X)_{|D}$, with $X\in D$, from \eqref{3.12}
it is deduced that $X\in D\cap D^\perp$. Using \eqref{eq:symplectic:sub},
$\bar{\Lambda}^\sharp(\bar{\alpha})=X=0$, so ii) in Theorem \ref{thm:inverse} holds.

Finally, a direct computation shows that the generalized distribution
\[
p\in M\to \mathcal{F}(p)=\bar{\Lambda}^\sharp (S(p))\subseteq T_pM,
\]
associated with the pair $(S,\bar{\Lambda}^\sharp )$ is just $D$, which is involutive by hypothesis. Morever,
for every leaf $L$, the $\R^k$-valued two form $\bar{\omega}_L$ is just the pull-back to $L$ of $\bar{\omega}$. $\bar{\omega}_L$ is clearly closed and the nondegeneracy \eqref{poly-cond} is
a consequence of \eqref{eq:symplectic:sub}.

Thus, using Theorem \ref{thm:inverse}, $(M,S,\bar{\Lambda}^\sharp)$ is a poly-Poisson manifold with polysymplectic foliation $D$.
\end{example}


\section{Poly-Poisson structures and reduction of polysymplectic structures}\label{sec:reduction}

As it is well known, one can obtain Poisson manifolds from symplectic manifolds through a reduction 
procedure. Let us recall some details of the process. Let $(M,\omega )$ be a symplectic manifold 
and $G$ be a Lie group acting freeely and properly on $M$.
If the action $\Phi \colon G\times M\to M$ preserves the symplectic structure (i.e., $\Phi 
_g^*\omega =\omega$, for all $g\in G$) then the orbit space $M/G$ inherits a Poisson structure 
$\Lambda$ in such a way
that the bundle projection $\pi \colon M\to M/G$ is a Poisson map, that is, $\Lambda ^\sharp (\pi 
(x)) =T_x\pi \circ (\omega ^\flat (x))^{-1}\circ T^ *_x\pi$, for $x\in M$. A particular example of 
this situation is the following one. If the action of $G$ on $Q$ is free and proper then it can be 
lifted to an action on $T^*Q$ which preserves the canonical symplectic structure $\omega _Q$. The 
corresponding Poisson structure on $T^*Q/G$ is the so-called \emph{Poisson-Sternberg structure}, 
and it may be seen as the linear Poisson structure associated to the Atiyah algebroid on $TQ/G$ 
(see \cite{Ma}).

We have introduced poly-Poisson structures to play the same role with respect to polysymplectic 
manifolds, so it is expected that the reduction of a polysymplectic
manifold is a poly-Poisson manifold. The next theorem asserts that this is true under certain 
regularity conditions.

\begin{theorem}\label{reduction}
Let $(M,\omega_1,\ldots,\omega_k)$ be a polysymplectic manifold and $G$ be a Lie group that acts over $M$
$(\Phi:G\times M\rightarrow M)$ freely, properly and satisfies $\Phi_g^*\omega_i=\omega_i$, for every $g\in G$ and $i\in\{1,\ldots ,k\}$.
Suppose that $V\pi$ is the vertical bundle of the principal bundle projection $\pi:M\rightarrow M/G$ and that
the following conditions hold:
\begin{itemize}
\item[i)] $\textrm{Im }(\bar{\omega}^{\flat})\cap {(V\pi)^{\circ}}^k$ is a vector subbundle,
where ${(V\pi)^{\circ}}^k=(V\pi)^{\circ}\times\stackrel{(k}\ldots\times {(V\pi)^{\circ}}$.
\item[ii)]
$(\bar{\omega}^{\flat}(p))^{-1}\left(({(V_p\pi)^{\perp}})^{\circ k}\cap (V_p \pi)^{\circ k}\cap S(p) \right) \subseteq V_p\pi$, where $(V\pi)^\perp$ is the polysymplectic
orthogonal, that is, $(V\pi)^\perp =\bigcap _{i=1}^k (\omega^\flat_i)^{-1}((V\pi )^\circ )$ and
$S$ is the image of $\bar{\omega}^{\flat}$, that is, $S=\textrm{Im
}(\bar{\omega}^{\flat})$.
\end{itemize}
Then, there is a natural k-poly-Poisson structure on $M/G$.
\end{theorem}
\begin{pf}
First of all, denote by $(T^*_p\pi)_k^1: (T^1_k)_{\pi (p)}^*(M/G)\rightarrow (T^1_k)_p^*M $ to be $k$ copies of $T^*_p\pi$.

From the fact that the $2$-forms $\omega_1,\ldots,\omega_k$ are $G$-invariant,
$S$ is invariant under the lifted action
\[
\begin{array}{rccl}
\Phi^{(T^1_k)^*}:& G\times (T^1_k)^*M &\rightarrow& (T^1_k)^*M\\
& (g,(\alpha_1,\ldots,\alpha_k))&\rightarrow &(\Phi^*_{g^{-1}}\alpha_1,\ldots,\Phi^*_{g^{-1}} \alpha_k).
\end{array}
\]
In addition, $\pi:S\rightarrow M$ is $G$-equivariant.

Now, for every $p\in M$, define the subspace $\hat{S}(\pi (p))$ of $(T^1_k)_{\pi (p)}^*(M/G)$
given by
\
\begin{equation}\label{equ:stilde}
\hat{S}(\pi (p)):=\{\hat{\alpha}=(\alpha_1,\ldots,\alpha_k)\in (T^1_k)_{\pi(p)}^*(M/G)\; |\; 
(T_p^*\pi)^1_k(\hat{\alpha})\in S(p) \}.
\end{equation}

Note that the $G$-invariant character of $S$ implies that 
\[
(T^*_{gp}\pi )^1_k(\hat{\alpha})\in S(gp)\iff (T^*_{p}\pi )^1_k(\hat{\alpha})\in S(p),\quad \forall g\in G.
\]
Thus, the definition of $\hat{S}(\pi (p))$ does not depend on the chosen point $p\in M$.

In fact, since $(T^*_p\pi)^1_k \left( (T^1_k)_{\pi(p)}^*(M/G)\right)={(V_p\pi)^{\circ}}^k$, it follows that $(T^*_p\pi)^1_k (\hat{S}(\pi (p)))=S(p)\cap {(V_p\pi)^{\circ}}^k$, for all $p\in M$.
In particular, $\hat{S}(\pi (p))\cong S(p)\cap {(V_p\pi)^{\circ}}^k$.

Now, using hypothesis i) in the theorem, we have that $S\cap {(V\pi)^{\circ}}^k$ is a $G$-invariant
vector subbundle of $(T^1_k)^*M$. Therefore, $\hat{S}$ is isomorphic to the quotient vector 
subbundle 
\[
\left(S\cap {(V\pi)^{\circ}}^k\right)/G=\left(\textrm{Im }(\bar{\omega}^{\flat})
\cap {(V\pi)^{\circ}}^k\right)/G .
\]


Next, define the vector bundle morphism $\hat{\Lambda}^{\sharp}:\hat{S}\rightarrow T(M/G)$ as
$\hat{\Lambda}^{\sharp}(\pi (p)):=T_p\pi \circ (\bar{\omega}^{\flat}(p))^{-1} \circ(T^*_p\pi)^1_k{}_{|\hat{S}(\pi (p))}$, i.e.,
$\hat{\Lambda}^\sharp(\pi (p))$ makes the following diagram commutative
\[
\xymatrix{
T_pM\ar[d]_{T_p\pi}& \ar[l]_{(\bar{\omega}^{\flat}(p))^{-1}}  S(p) \\
T_{\pi(p)}(M/G)& \hat{S}(\pi(p)) \ar[l]^{\hat{\Lambda}^\sharp(\pi (p))} \ar[u]_{(T^*_p\pi)^1_k}.
}
\]
Note that the $G$-invariant character of $\hat{S}$ and of the 2-forms $\omega_i$ implies that the map $\hat{\Lambda}^\sharp(\pi (p))$ is well defined (it is independent of the base point $p\in M$).

Let us prove that $(\hat{S},\hat{\Lambda}^{\sharp})$ is a poly-Poisson structure on $M/G$.

First, condition i) of Definition \ref{def:poly-poisson} is a direct consequence of the skew-symmetric character of the monomorphism $\bar{\omega}^{\flat}$.

In order to prove condition ii) in Definition \ref{def:poly-poisson}, take $\hat{\alpha}\in \hat{S}(\pi (p))$ such that
$\hat{\alpha}_{|\textrm{Im }(\hat{\Lambda}^{\sharp}(\pi(p)))}=0$. From the definition of $\hat{\Lambda}^\sharp$
we have
\begin{equation}\label{eq:reduced:dist}
\begin{array}{rcl}
\textrm{Im }(\hat{\Lambda}^{\sharp}(\pi (p)))&=&\left(T_p\pi\circ (\bar{\omega}^{\flat}(p))^{-1} \circ(T^*_p\pi)^1_k \right)(\hat{S}(\pi (p)))\\
&=&\left(T_p\pi \circ (\bar{\omega}^{\flat}(p))^{-1}\right) (S(p)\cap ({V_p\pi})^{\circ k})= (T_p\pi )({V_p\pi})^{\perp}.
\end{array}
\end{equation}
Thus, $\hat{\alpha}_{|\textrm{Im }(\hat{\Lambda}^{\sharp}(\pi(p)))}=0$ is equivalent to
$[(T^*_p\pi)^1_k(\hat{\alpha})]_{|{(V_p\pi)}^{\perp}}=0$, and, as a consequence,
$\hat{\alpha}\in \hat{S}$ with $\hat{\alpha}_{|\textrm{Im }(\hat{\Lambda}^{\sharp}(\pi(p)))}=0$ if and only
if
\[
(T^*_p\pi)^1_k(\hat{\alpha}) \in
\left(({(V_p\pi)^{\perp}})^{\circ k}\cap (V_p \pi)^{\circ k}\cap S(p) \right) .
\]
Therefore, Definition \ref{def:poly-poisson} ii) is the relation
\[
(T_p \pi \circ (\bar{\omega}^{\flat}(p) )^{-1})  \left(({(V_p\pi)^{\perp}})^{\circ k}\cap (V_p \pi)^{\circ k}\cap S(p) \right)  =0,
\]
and using that $V_p\pi=\textrm{Ker }(T_p\pi)$, this is equivalent to
\[
(\bar{\omega}^{\flat}(p))^{-1}\left(({(V_p\pi)^{\perp}})^{\circ k}\cap (V_p \pi)^{\circ k}\cap S(p) \right) \subseteq V_p\pi,
\]
which is just the second hypothesis in the theorem.

Finally, we will prove that the integrability condition \eqref{inte} holds for the couple $(\hat{S}, \hat{\Lambda}^\sharp)$.

Denote by $\Lambda^\sharp \colon S\subseteq (T^1_k)^* M\to TM$ the isomorphism $(\bar{\omega}^\flat)^ {-1}$ and by
\[
(\pi ^1_k)^*\colon \Gamma ((T^1_k)^*(M/G))\to \Gamma ((T^1_k)^*M)
\]
the monomorphism of $C^\infty (M/G)$-modules induced by the projection $\pi$ between 
the spaces $\Gamma ((T^1_k)^*(M/G))$ and $\Gamma ((T^1_k)^*M)$ of sections of the vector bundles
$(T^1_k)^*(M/G)\to M/G$ and $ (T^1_k)^*M\to M$, respectively. Then, if $\hat{\alpha},\hat{\beta}
\in \Gamma ((T^1_k)^*(M/G))$ it follows that the vector fields 
$\Lambda^\sharp ((\pi ^1_k)^*\hat{\alpha})$ and $\Lambda^\sharp ((\pi ^1_k)^*\hat{\beta})$ are
$\pi$-projectable on the vector fields $\hat{\Lambda}^\sharp (\hat{\alpha})$ and  $\hat{\Lambda}^\sharp (\hat{\beta})$, respectively. Moreover, since the couple $(S,\Lambda^\sharp )$ is a poly-Poisson structure on $M$, we deduce that
\[
[ \Lambda^\sharp ((\pi ^1_k)^*\hat{\alpha}),\Lambda^\sharp ((\pi ^1_k)^*\hat{\beta}) ]=
(\Lambda ^\sharp \circ (\pi ^1_k)^* ) \left({\mathcal{L}}_{\hat{\Lambda}^{\sharp}(\hat{\alpha})}
\hat{\beta}-{\mathcal{L}}_{\hat{\Lambda}^{\sharp}(\hat{\beta})}\hat{\alpha}-
\d(\hat{\beta}(\hat{\Lambda}^{\sharp}(\hat{\alpha})))\right).
\]
Thus, if we project on $M/G$ using $\pi$, we conclude that
\[
[ \hat{\Lambda}^{\sharp}(\hat{\alpha}),\hat{\Lambda}^{\sharp}(\hat{\beta}) ]=
\hat{\Lambda}^\sharp  \left({\mathcal{L}}_{\hat{\Lambda}^{\sharp}(\hat{\alpha})}
\hat{\beta}-{\mathcal{L}}_{\hat{\Lambda}^{\sharp}(\hat{\beta})}\hat{\alpha}-
\d(\hat{\beta}(\hat{\Lambda}^{\sharp}(\hat{\alpha})))\right).
\]
\end{pf}

\begin{example}[Reduction of the cotangent bundle of $k$-covelocities
associated with a Lie group]\label{ex:red:group}

Let $G$ be a Lie group. The cotangent bundle of $k$-covelocities of $G$, $(T^1_k)^*G$ has a canonical
$k$-polysymplectic structure (Example \ref{canexample}). In addition, the action of $G$ on itself by left translations
\[
\begin{array}{rccl}
L\colon & G\times G & \longrightarrow & G \\
&   (g,h) & \rightarrow &L(g,h)=g\cdot h
\end{array}
\]
can be lifted to an action $L^{(T^1_k)^*}$ on $(T^1_k)^*G$ by
\begin{equation}\label{6}
\begin{array}{rccl}
L^{(T^1_k)^*}\colon & G\times(T^1_k)^*G & \longrightarrow & (T^1_k)^*G \\
&    (g,(\alpha_1,\ldots,\alpha_k))&\rightarrow &(L^*_{g^{-1}}\alpha_1,\ldots, L^*_{g^{-1}} \alpha_k).
\end{array}
\end{equation}

If we consider the left trivialization of the tangent bundle to $G$, $TG\cong G\times {\mathfrak g}$, and the corresponding trivialization
$(T^1_k)^*G\cong G\times \mathfrak{g}^*_k$, $\mathfrak{g}^*_k=\mathfrak{g}^*\times\stackrel{(k}{\ldots}\times \mathfrak{g}^*$, the action \eqref{6} can be written as
\begin{equation}\label{action}
\begin{array}{rccl}
L^{(T^1_k)^*}\colon & G\times \mathfrak{g}^*_k & \longrightarrow & G\times \mathfrak{g}^*_k \\
&    (g,(h,\mu_1,\ldots,\mu_k))&\rightarrow &(g\cdot h,\mu_1,\ldots, \mu_k),
\end{array}
\end{equation}
and the associated principal bundle $\pi\colon (T^1_k)^*G\rightarrow ((T^1_k)^*G)/G$ can be identified with the trivial principal bundle 
$\pi\colon G\times \mathfrak{g}^*_k\rightarrow
\mathfrak{g}^*_k$. Thus, the vertical bundle takes the expression
\begin{equation}\label{eq:vertical}
V\pi_{(g,\mu_1,\ldots,\mu_k)}=\{(X,0,\ldots,0)\textrm{ such that }X\in
T_gG\}. 
\end{equation}
Therefore, its annihilator is
\begin{equation}\label{eq:vert:annihilator}
({V\pi})_{(g,\mu_1,\ldots,\mu_k)}^{\circ}=\{(0,\xi_1,\xi_2,\ldots,\xi_k)
\,\, | \,\, \xi_1,\ldots,\xi_k\in\mathfrak{g}\}.
\end{equation}

We will show that the action of $G$ on $(T^1_k)^*G\cong G\times \mathfrak{g}^*_k$, endowed with the canonical
polysymplectic structure, satisfies the hypotheses of Theorem \ref{reduction}.

First of all we are going to compute explicitly the polysymplectic structure on $G\times
\mathfrak{g}^*_k$. For this purpose, it is first described the canonical symplectic structure $\omega$ on
$T^*G\cong G\times \mathfrak{g^*}$ (see for instance, \cite{OR-2004}). Using the left translation
by $g\in G$ we can identify $T_gG$ and $T^*_gG$ with $\mathfrak{g}$ and $\mathfrak{g}^*$, 
respectively. Thus, we have that $T_{(g,\mu)}(G\times
\mathfrak{g}^*)\cong \mathfrak{g}\times  \mathfrak{g}^*$ and $T_{(g,\mu)}^*(G\times \mathfrak{g}^*)\cong \mathfrak{g}^*\times \mathfrak{g}$ and, under
these identifications, it follows that
\begin{equation}
\hspace{.1cm}
\begin{array}{ccl}
\kern20pt\omega_{(g,\mu)}\kern-2pt: T_{(g,\mu)}(G\times \mathfrak{g}^*)\kern-2pt \times \kern-2pt T_{(g,\mu)}(G\times \mathfrak{g}^*)\cong (\mathfrak{g}\kern-2pt\times  \kern-2pt\mathfrak{g}^*)\kern-2pt \times \kern-2pt(\mathfrak{g}\kern-2pt\times \kern-2pt \mathfrak{g}^*)&\kern-5pt\longrightarrow &\kern-5pt\mathbb{R}\\
\left((\xi_1,\tau_1),(\xi_2,\tau_2)\right)&\kern-5pt\rightarrow &\kern-5pt\omega_{(g,\mu)}((\xi_1,\tau_1),(\xi_2,\tau_2))\\
\noalign{\medskip} &&\kern-20pt=\tau_2(\xi_1)- \tau_1(\xi_2)+\mu([\xi_1,\xi_2]).
\end{array}
\end{equation}
Therefore,
\begin{equation}\label{1}
\begin{array}{ccl}
\omega_{(g,\mu)}^{\flat}:  T_{(g,\mu)}(G\times \mathfrak{g}^*)\cong \mathfrak{g}\times  \mathfrak{g}^*&\longrightarrow &T_{(g,\mu)}^*(G\times \mathfrak{g}^*)\cong \mathfrak{g}^*\times  \mathfrak{g} \\ \noalign{\medskip}
(\xi,\tau)&\rightarrow &\omega_{(g,\mu)}^{\flat}(\xi,\tau)=(ad^*_{\xi}\mu -\tau,\xi),
\end{array}
\end{equation}
and, as a consequence, $\left(\omega_{(g,\mu)}^{\flat}\right)^{-1}(\tau,\xi)=(\xi,ad^*_{\xi}\mu -\tau)$, for 
$(\tau,\xi)\in \mathfrak{g}^*\times \mathfrak{g}\cong T_{(g,\mu)}^*(G\times \mathfrak{g}^*)$.

Now, denote by $\bar{\omega}^{\flat}\colon T((T^1_k)^*G)\to (T^1_k)^*((T^1_k)^*G)$
the vector bundle morphism
associated with the polysymplectic structure on $(T^1_k)^*G\cong G\times \mathfrak{g}^*_k$.
Then, from Eq. \eqref{1} and definition of the canonical $k$-polysymplectic
structure on $(T^1_k)^*G$ we have
\begin{equation}\label{2}
\begin{array}{rcl}
\bar{\omega}_{(g,\mu_1,\ldots,\mu_k)}^{\flat}(\xi,\tau_1,\ldots, \tau_k) &=&
\Big( (ad^*_{\xi}\mu_1-\tau_1,\xi,0,\ldots, 0),(ad^*_{\xi}\mu_2-\tau_2,0,\xi,0,\ldots,0),\ldots,
\\ & & \kern5pt , \ldots , (ad^*_{\xi}\mu_k-\tau_k,0,0,\ldots,{\xi})\Big ) ,
\end{array}
\end{equation}
for $(\xi,\tau_1,\ldots, \tau_k)\in\mathfrak{g}\times\mathfrak{g}^*_k$. In addition, using Eq. \eqref{2}, it is deduced that 
\begin{equation}\label{4}
(S\cap{(V\pi)^{\circ}}^k)_{(g,\mu_1,\ldots,\mu_k)}=\{\left((0,\xi,0,\ldots, 0),(0,0,\xi,0,\ldots,0),\ldots,(0,0,0,\ldots,\xi)\right)\,|\,
\xi\in\mathfrak{g}\}
\end{equation}
which implies that hypothesis i) in Theorem \ref{reduction} holds.


On the other hand, from Eq. \eqref{1} and \eqref{2}, we obtain
\[
(V\pi)^{\perp}_{(g,\mu_1,\ldots,\mu_k)}=\{(\xi,ad^*_{\xi}\mu_1,ad^*_{\xi}\mu_2,\ldots,ad^*_{\xi}\mu_k) \, |\, \xi \in \mathfrak{g} \}
\]
and, as a consequence,
\[
((V\pi)^{\perp}_{(g,\mu_1,\ldots,\mu_k)})^{\circ}=\{(ad^*_{\xi_1}\mu_1+\ldots +
ad^*_{\xi_k}\mu_k,\xi_1,\xi _2,\ldots,\xi _k) \, | \, \xi_i\in\mathfrak{g},
\ i \in\{1,\ldots,k\} \}.
\]
Therefore, we deduce that $({(V\pi)^{\perp}}^{\circ})^k\cap{(V\pi)^{\circ}}^k\cap \textrm{Im
}(\bar{\omega}^{\flat})) _{(g,\mu_1,\ldots,\mu_k)}$ is just
\begin{equation}\label{5}
\{\left((0,\eta,0,\ldots,0),(0,0,\eta,\ldots,0),\ldots,(0,0,0,\ldots,\eta)\right)\, | \, \eta\in\mathfrak{g}, \,
ad^*_{\eta }\mu_i=0, \ i\in\{1, \ldots,k\}\}.
\end{equation}
Thus, given an element of 
$({(V\pi)^{\perp}}^{\circ})^k\cap{(V\pi)^{\circ}}^k\cap \textrm{Im }(\bar{\omega}^{\flat}))_{(g,\mu_1,\ldots,\mu_k)}$,
\[
\begin{array}{rcl}
{(\bar{\omega}^{\flat})}^{-1}_{(g,\mu_1,\ldots,\mu_k)}\left((0,\eta,0,\ldots,0),(0,0,\eta,\ldots,0),\ldots,(0,0,0,\ldots,\eta)\right)
&\kern-7pt =&\kern-7pt (\eta,ad^*_{\eta}\mu_1,ad^*_{\eta}\mu_2,\ldots,ad^*_{\eta}\mu_k) \\
&\kern-7pt =&\kern-7pt (\eta,0,0,\ldots,0)
\end{array}
\]
which implies that this vector is vertical. So, the hypothesis ii) of Theorem \ref{reduction} is satisfied.

Now, let us give an explicit description of $\hat{S}$ and $\bar{\Lambda}^\sharp$. 

\textbf{Computation of $\hat{S}$:} Recall by Theorem \ref{reduction}, that $\hat{S}$ is the 
subbundle of $\mathfrak{g}^*_k$ given by $\hat{S}\cong\left(S\cap 
{(V\pi)^{\circ}}^k\right)/G=\left(\textrm{Im
}(\bar{\omega}^{\flat})\cap {(V\pi)^{\circ}}^k\right)/G $. By \eqref{4} and the definition of the 
action \eqref{action} is obvious that the fiber of $\hat{S}$ at a point $\mu=(\mu_1,\ldots,
\mu_k)\in\mathfrak{g}^*_k$ is given by
\[
\hat{S}_{\mu}=\{((\xi,0,\ldots,0),(0,\xi,\dots,0),\ldots,(0,0,\ldots,\xi)) \; | \; \xi\in\mathfrak{g}\}
\]

\textbf{Computation of ${\hat{\Lambda}}^{\sharp}$: } Remember that ${\hat{\Lambda}}^{\sharp}$ is given by
$T_{(g,\mu)}\pi \circ (\bar{\omega}^{\flat}_{(g,\mu)})^{-1} \circ (T^*_{(g,\mu)}\pi )^1_k$ 
So on $\mu=(\mu_1,\ldots,\mu_k)\in
\mathfrak{g}^*_k$, given $((\xi,0,\ldots,0),(0,\xi,\dots,0),\ldots,(0,0,\ldots,\xi))\in \hat{S}_{\mu}$
\[
\begin{array}{l}
{\hat{\Lambda}}^{\sharp}_{\mu}((\xi,0,\ldots,0),(0,\xi,\dots,0),\ldots,(0,0,\ldots,\xi))\\ \noalign{\medskip}
=\left(T_{(g,\mu)}\pi \circ (\bar{\omega}^{\flat}_{(g,\mu)})^{-1} \circ (T^*_{(g,\mu)}\pi )^1_k \right)((\xi,0,\ldots,0),(0,\xi,\dots,0),\ldots,(0,0,\ldots,\xi))\\ \noalign{\medskip}
=\left(T_{(g,\mu)}\pi \circ (\bar{\omega}^{\flat}_{(g,\mu)})^{-1} \right)\left((0,\xi,0,\ldots,0),(0,0,\xi,\ldots,0),\ldots,(0,0,0,\ldots,\xi)\right))\\
\noalign{\medskip} =T_{(g,\mu)}\pi  (\xi,ad^*_{\xi}\mu_1,ad^*_{\xi}\mu_2,\ldots,ad^*_{\xi}\mu_k)\\ \noalign{\medskip}
=(ad^*_{\xi}\mu_1,ad^*_{\xi}\mu_2,\ldots,ad^*_{\xi}\mu_k)
\end{array}
\]
where we used \eqref{2}.

Note that this poly-Poisson structure coincides with  the one obtained using the construction of
Example \ref{ex:whitney-algebroid} (see \eqref{eq:algebra:S} and \eqref{eq:algebra:lambda} in 
Remark \ref{rmk:whitney}). Moreover, as a consequence, we have that
the polysymplectic leaves are just the orbits of the $k$-coadjoint action.

\end{example}

\begin{example}[Reduction of the cotangent bundle of $k$-covelocitites
associated with the total space of a principal bundle]\label{sub:reduc:whitney}

Let $Q$ be a manifold  of dimension $n$ endowed with
a free and proper action of a Lie group $G$, $\Psi:G\times
Q\rightarrow Q$. Let $\tilde\pi \colon Q\to Q/G$ be the corresponding
principal bundle projection. As we did in the previous example, we can define the lifted action
\[
\begin{array}{rccl}
\Psi^{(T^1_k)^*}: & G\times(T^1_k)^*Q & \longrightarrow & (T^1_k)^*Q \\
&    (g,(\alpha_1,\ldots,\alpha_k))&\rightarrow &(\Psi^*_{g^{-1}}\alpha_1,\ldots,\Psi^*_{g^{-1}} \alpha_k)
\end{array}
\]
and this action still is free and proper. Thus, we have the principal bundle $\pi:(T^1_k)^*Q\rightarrow ((T^1_k)^*Q)/G $ and, moreover, $((T^1_k)^*Q)/G\cong T^*Q/G\oplus\stackrel{(k}{\ldots}\oplus T^*Q/G$.

Let us see that the action of $G$ on the polysymplectic manifold $(T^1_k)^*Q$ satisfies the 
conditions of Theorem \ref{reduction}. It is clear that the lifted action preserves the 
polysymplectic structure. In fact, since the lifted action
of $G$ on $T^*Q$ preserves the canonical symplectic structure $\omega _Q$, the action 
$\Psi^{(T^1_k)^*}$ preserves the polysymplectic structure on $(T^1_k)^*Q$.

Next, let $U$ be an open subset of $Q/G$ such that $\tilde\pi^{-1}(U)$
is a trivializing open subset of $Q$, that is
$\tilde\pi^{-1}(U)\cong U\times G$ and the principal action of $G$ on
$\tilde\pi^{-1}(U)\cong U\times G$ is given by $\Psi(g,(u,h))=(u,gh)$,
for $g\in G$ and $(u,h)\in U\times G$. As a consequence, $(T^1_k)^* (\tilde\pi^{-1}(U))\cong (T^1_k)^*U \times (T^1_k)^*G $.
Moreover, if $\bar{\omega}_{U}$ and $\bar{\omega}_{G}$ are the canonical
polysymplectic forms on $(T^1_k)^*U$ and $(T^1_k)^*G$ respectively then the polysymplectic form on
$(T^1_k)^*(\tilde\pi^{-1}(U))\cong  (T^1_k)^*(U)\times
(T^1_k)^*G$, is given by $\bar{\omega}_{U}\times\bar{\omega}_{G}$.

On the other hand, under the identification $(T^1_k)^*(\tilde\pi^{-1}(U))\cong  (T^1_k)^*(U)\times
(T^1_k)^*G$ the lifted action of $G$ on $(T^1_k)^*
(\tilde\pi^{-1}(U))$ is given by
\begin{equation}\label{ecu:action}
\Psi^{(T^1_k)^*}(g,(\widetilde{\alpha},\widetilde{\gamma}))=
(\widetilde{\alpha},L^{(T^1_k)^*}(g,\widetilde{\gamma})), 
\end{equation}
for $g\in G$ and $(\widetilde{\alpha},\widetilde{\gamma})\in (T^1_k)^*(U)\times
(T^1_k)^*G$. Therefore, the quotient space $(T^1_k)^* (\tilde\pi^{-1}(U))/G$ is just  $(T^1_k)^*(U)\times \mathfrak{g}^*_k$, where it has been used the identification
$(T^1_k)^*G\cong G\times\mathfrak{g}^*_k$ as in Example \ref{ex:red:group}.

Now, we will prove that the polysymplectic structure on $(T^1_k)^*(\tilde\pi^{-1}(U))
\kern-1pt\cong\kern-1pt  (T^1_k)^*(U)\times
(T^1_k)^*G\cong (T^1_k)^*U\times (G\times \mathfrak{g}^*_k)$ and the
action of $G$ on $(T^1_k)^*(\tilde{\pi}^{-1}(U))$ given by
\eqref{ecu:action}, satisfy the hypotheses of Theorem \ref{reduction}.

In order to check these hypotheses, we split the problem in two parts,
the side of $(T^1_k)^* U$ and the one of $(T^1_k)^*G$. Using that the
intersection of the vertical bundle to $\pi$ with $T((T^1_k)^*U)$ is
trivial and the results in the previous Example
\ref{ex:red:group}, we deduce that the hypotheses of Theorem
\ref{reduction} hold.

On the other hand, from \eqref{equ:stilde}, it follows that:
\[
\hat{S}(\pi(\overline{\alpha}))=\{(\hat{\alpha}_1,{\ldots},\hat{\alpha}_k)\in(T^1_k)^*_{\pi(\overline{\alpha})}(T^*Q/G\oplus\stackrel{(k}{\ldots}\oplus
T^*Q/G ) \ | \
(T^*_{\overline{\alpha}}\pi)^1_k(\hat{\alpha}_1,{\ldots}
,\hat{\alpha}_k)\in S(\overline{\alpha})\},
\]
for $\overline{\alpha}\in (T^1_k)^*_qQ$, where
$S=\textrm{Im}(\overline{\omega}^{\flat})$ and $\overline{\omega}$ is
the canonical polysymplectic structure on $(T^1_k)^*Q$.

This implies that $\hat{S}$ is given by \eqref{eq:dist:whitney}. Thus,
$(\hat{\alpha}_1,\ldots,\hat{\alpha}_k)\in
\hat{S}(\pi(\overline{\alpha}))$ if
\[
\hat{\alpha}_A=(T^*_{\overline{\alpha}}\tilde{\pi}_A)(\tilde{\alpha}_A),
\textrm{ with }\tilde{\alpha}_A\in T^*_{\tilde{\pi}_A(\overline{\alpha})}(T^*Q/G)
\]
and
\[
\tilde{\alpha}_A\circ ^{\textrm{v}}_{\tilde{\pi}
  _A(\overline{\alpha})}= \tilde{\alpha}_B\circ
^{\textrm{v}}_{\tilde{\pi}_B(\overline{\alpha})}\textrm{, for all $A$
  and $B$},
\]
$\tilde{\pi}_A:T^*Q/G\oplus\stackrel{(k}{\ldots}\oplus
T^*Q/G\to T^*Q/G$ being the canonical projection over the $A$th-factor.

In order to compute ${\hat{\Lambda}}^{\sharp}$, we have to take into
account that, from Theorem
\ref{reduction}, ${\hat{\Lambda}}^{\sharp}_{\tilde\pi(\widetilde{\alpha}_1,\ldots
  ,\widetilde{\alpha}_k)}=T_{(\alpha_1,\ldots,\alpha_k)}\pi \circ
(\overline{\omega}^{\flat}_ {(\alpha_1,\ldots,\alpha_k)})^{-1}\circ
(T^*_{(\alpha_1,\ldots,\alpha_k)}\pi)^1_k$. Moreover, if
$\tilde{\alpha}_A\in T^*_{\tilde{\pi}_A(\overline{\alpha})}(T^*Q/G)$,
for all $A$, and $\pi_A:T^*Q\oplus\stackrel{(k}{\ldots}\oplus
T^*Q\to T^*Q$ is the canonical projection then
\[
\hat{\Lambda}^{\sharp}_{\pi(\overline{\alpha})}((T^*_{\overline{\alpha}}\pi_1)(\tilde{\alpha}_1),\ldots,
(T^*_{\overline{\alpha}}\pi_k)(\tilde{\alpha}_k))=(T_{\overline{\alpha}}\pi\circ
(\overline{\omega}_{\overline{\alpha}}^{\flat})^{-1}\circ
T^*_{\overline{\alpha}}\pi)((T^*_ {\overline{\alpha}}\pi_1)(\tilde{\alpha}_1),\ldots,(T^*_ {\overline{\alpha}}\pi_k)(\tilde{\alpha}_k)).
\]

The poly-Poisson structure on $T^*Q/G\oplus\ldots\oplus T^*Q/G$ is just the one associated to the linear
Poisson structure on $T^*Q/G$ associated to the Atiyah algebroid $TQ/G$ (see Example
\ref{ex:whitney-algebroid}).
\end{example}
\begin{remark}
Let $Q$ be an arbitrary manifold of dimension $n$. Then, the frame bundle $\pi\colon LQ\to Q$ is endowed with a polysymplectic structure (see Example \ref{ex:frame:bundle}) and, moreover, it is an open subset of $(T^1_n)^*Q$, and the restriction of the polysymplectic structure on $(T^1_n)^*Q$ is just $-\d\vartheta$ (see Remark \ref{rmk:relation}).

If $G$ is a Lie group acting freely and properly on $Q$ then $G$ acts on $(T^1_n)^*Q$. In addition, $LQ$ is
stable under this lifted action (note that, since $\Psi _g\colon Q\to Q$ is diffeomorphism for any $g\in G$,
$T\Psi _g\colon T_qQ\to T_{\Psi _g(q)}Q$ is an isomorphism).
Thus, from Theorem \ref{reduction}, the quotient space $LQ/G$ is a poly-Poisson structure. But a simple
computation shows that $LQ/G$ is just the set
\[
\{ u\colon \R^n\to (TQ/G)_{[q]} \textrm{ linear isomorphism} \, | \, [q] \in Q/G \},
\]
that is, $LQ/G$ is  $L(TQ/G)$, the frame bundle associated to the Atiyah algebroid $TQ/G$. The computations in Example \ref{sub:reduc:whitney} allow to conclude that the poly-Poisson structure on $LQ/G$ obtained from the reduction procedure is the  poly-Poisson structure on $L(TQ/G)$ from Example \ref{ex:poly:frame}.
\end{remark}


\section*{Appendix: Lie algebroids and fiberwise linear Poisson structures} \label{algebroides}

\newcounter{apendice}
\setcounter{apendice}{0}
\setcounter{equation}{0}
\renewcommand{\thesection}{A}

We will review some basic facts on Lie algebroids and fiberwise linear Poisson structures (for more,
details, see \cite{Ma}).

A \emph{Lie algebroid} is a real vector bundle $\tau _E\colon E\to Q$ of rank $n$
over a manifold $Q$ of dimension $m$ such that the space of sections $\Gamma (E)$
admits a Lie algebra structure $\lcf\cdot,\cdot\rcf_E$ and, moreover, there exists a vector bundle map
$\rho _E\colon E\to TQ$, \emph{the anchor map}, such that if we also
denote by $\rho _E\colon \Gamma(E)\to {\X }(Q)$ the corresponding morphism of
$C^\infty(Q)$-modules then
\[
\lcf X,fY\rcf _E=f\lcf X,Y\rcf _E+ \rho_E(X)(f)Y,
\]
for $X,Y\in \Gamma(E)$ and $f\in C^\infty(Q)$.

The previous conditions imply that $\rho _E$ is a Lie algebra morphism, that is,
\[
\rho _E \lcf X, Y\rcf _E=[\rho _E(X),\rho _E(Y)], \qquad \textrm{ for }X, Y\in \Gamma (E).
\]
Some examples of Lie algebroids are the tangent bundle to a manifold $Q$ and a real
Lie algebra of finite dimension.

Let $(\lcf\cdot,\cdot\rcf _E,\rho _E)$ be a Lie algebroid structure on a real vector bundle
$\tau _E\colon E\to Q$. If $(q^i)$ are local coordinates on an open subset $U\subseteq Q$ and
$\{e_\alpha \}$  is a local basis of $\Gamma (E)$ then
\[
\rho _E(e_\alpha )=\rho_\alpha ^i\frac{\partial }{\partial q^i},\qquad \lcf e_\alpha ,e_\beta \rcf _E
={\mathcal C }_{\alpha \beta}^\gamma e_\gamma,
\]
with $\rho_\alpha ^i, {\mathcal C}_{\alpha \beta}^{\gamma}\in C^\infty (U)$
\emph{the local structure functions} of $E$ with respect to the local coordinates $(q^i)$
and the local basis $\{ e_\alpha \}$.

The local structure functions satisfy the \emph{the local structure equations} on $E$
\begin{equation}\label{estruc1}
\rho_\alpha ^i\frac{\partial \rho_\beta ^j}{\partial q^i}
-\rho_\beta ^i\frac{\partial \rho_\alpha ^j}{\partial q^i}= \rho_\gamma ^j{\mathcal
C}_{\alpha \beta}^\gamma , \qquad
\sum_{cyclic(\alpha, \beta,\gamma)}\Big ( \rho_{\alpha}^i\frac{\partial {\mathcal
C}_{\beta\gamma}^\nu }{\partial q^i} + {\mathcal C}_{\alpha \mu}^\nu {\mathcal
C}_{\beta\gamma}^\mu \Big )=0.
\end{equation}

Let $\tau_{E^*}\colon E^*\to Q$ be the dual bundle to a Lie algebroid $\tau _E\colon E\to Q$. Then,
$E^*$ admits a \emph{fiberwise linear Poisson bracket} $\{\cdot ,\cdot \}_{E^*}$ which is characterized
by the following relations,
\[
\begin{array}{c}
\{\widehat{X},\widehat{Y}\}_{E^*}=-\widehat{\lcf X,Y\rcf_E},\;\; \{ f\smalcirc \tau_{E^*},\widehat{Y}\}_{E^*}=
\rho _E(Y)(f)\smalcirc \tau _{E^*},\\
\{ f\smalcirc \tau_{E^*},g\smalcirc \tau_{E^*}\}_{E^*}=0,
\end{array}
\]
for $X,Y\in \Gamma(E)$ and $f,g\in C^\infty (Q)$. Here, if $Z\in\Gamma(E)$ we will denote by $\widehat{Z}\colon E^*\to \R$ the corresponding fiberwise linear function on $E^*$.

Let $(q^i)$ be local coordinates on an open subset $U\subseteq Q$ and $\{e_\alpha\}$ be a local basis
of $\Gamma (\tau_E^{-1}(U))$. Denote by $(q^i,p_\alpha )$ the corresponding local coordinates on
$E^*$. Thus,
\[
\{p_\alpha,p_\beta \}_{E^*}=-{\mathcal C}_{\alpha\beta}^\gamma p_\gamma,\; \{ q^i,p_\alpha \}_{E^*}=
\rho ^i_\alpha ,\; \{ q^i,q^j\}_{E^*}=0.
\]
Therefore, if $\Lambda _{E^*}$ is the corresponding Poisson 2-vector on $E^*$, it follows that
\begin{equation}\label{eq:linear:Poisson}
\Lambda _{E^*}^\sharp (\d q^i)=\rho ^i_\alpha \frac{\partial }{\partial p_\alpha},\;
\Lambda _{E^*}^\sharp (\d p_\alpha)=- ( \rho ^i_\alpha \frac{\partial }{\partial q^i} +{\mathcal C}_{\alpha \beta}^\gamma p_\gamma \frac{\partial}{\partial p_\beta})   .
\end{equation}

\begin{remark}
\begin{itemize}
\item[i)] The fiberwise linear Poisson structure on $T^*Q$ induced by the standard Lie algebroid structure
on $TQ$ is just the canonical Poisson structure on $T^*Q$ induced by the canonical symplectic structure
on $T^*Q$.
\item[ii)] If $\mathfrak{g}$ is a real Lie algebra of finite dimension then the linear Poisson structure
on $\mathfrak{g}^*$ is just the Lie-Poisson structure.
\end{itemize}
\end{remark}

\begin{remark}
Suppose that $\alpha _A\in E^*_q$ and $\widetilde{\alpha}_A\in T^*_{\alpha _A}E^*$, with $A\in \{1,2\}$
and $q\in Q$. Using \eqref{eq:linear:Poisson}, we deduce that
\begin{equation}\label{eq:relation:lin:Poisson}
(T_{\alpha _1}\tau _{E^*})(\Lambda ^\sharp _{E^*}(\widetilde{\alpha}_1))=(T_{\alpha _2}\tau _{E^*})(\Lambda ^\sharp _{E^*}(\widetilde{\alpha}_2))\iff \widetilde{\alpha}_1\smalcirc \, ^{\textrm{v}}_{\alpha_1}\smalcirc
\rho^*_{E}{}_{|T^*_qQ}=\widetilde{\alpha}_2\smalcirc \, ^{\textrm{v}}_{\alpha_2}\smalcirc
\rho^*_{E}{}_{|T^*_qQ},
\end{equation}
where $\rho_E^*\colon T^*Q\to E^*$ is the dual morphism of the anchor map $\rho _E\colon E\to TQ$
and $^{\textrm{v}}_{\alpha _A}\colon E^*_q\to T_{\alpha _A}E^*_q$ is the canonical isomorphism
between $E^*_q$ and $T_{\alpha _A}E^*_q$. We remark that
\[
^{\textrm{v}}_{\alpha _A} (e_\beta (q))=\frac{\partial}{\partial p_\beta}_{|\alpha _A}, \qquad \textrm{ for }
\beta \in \{1,\ldots ,n\}.
\]

\end{remark}

\textbf{Acknowledgments}
This work has been partially supported by MICINN (Spain)
Grants  MTM2009-13383,  MTM2010-21186-C02-01 and  MTM2009-08166-E, project "Ingenio
Mathematica" (i-MATH) No. CSD 2006-00032 (Consolider-Ingenio 2010), the project of the
Canary Islands government SOLSUB200801000238, the ICMAT Severo Ochoa
project SEV-2011-0087 and the European project IRSES-project ``Geomech-246981''. D.~Iglesias wishes to thank MICINN for a ``Ram\'on y Cajal" research
contract. M. Vaquero wishes to thank MICINN for a FPI-PhD Position.

\end{document}